\documentclass[a4paper, 11pt, oneside, reqno, svgnames]{amsart}

\usepackage[a4paper]{geometry}
\usepackage[utf8]{inputenc}
\usepackage[english]{babel}
\usepackage[T1]{fontenc}
\usepackage{textcomp}
\usepackage{txfonts}

\usepackage{apptools}

\usepackage[mathscr]{euscript}

\usepackage{tikz-cd}

\usepackage{todonotes}

\usepackage{amsmath}
\usepackage{amsfonts}
\usepackage{amsthm}
\usepackage{latexsym}
\usepackage{amssymb}
\usepackage{xcolor}
\definecolor{blue(munsell)}{rgb}{0.0, 0.5, 0.69}

\usepackage{hyperref}
\hypersetup{hypertexnames = false, bookmarksdepth = 2, bookmarksopen = true, colorlinks, linkcolor = black, citecolor = blue(munsell), urlcolor = blue(munsell), pdfstartview={XYZ null null 1}}
\usepackage{cleveref}

\makeatletter
\def\l@subsection{\@tocline{2}{0pt}{2pc}{5pc}{}}
\makeatother

\DeclareMathOperator{\Hom}{Hom}

\DeclareMathOperator{\Fun}{Fun}
\DeclareMathOperator{\Ob}{Ob}

\DeclareMathOperator{\Res}{Res}

\DeclareMathOperator{\Ind}{Ind}
\DeclareMathOperator{\Mod}{Mod}

\DeclareMathOperator{\Img}{Im}

\DeclareMathOperator{\hproj}{h-proj}
\DeclareMathOperator{\pretr}{pretr}
\DeclareMathOperator{\cone}{C}

\DeclareMathOperator{\compdg}{dgm}

\DeclareMathOperator{\dercomp}{\mathsf{D}}
\DeclareMathOperator{\dercompdg}{\mathsf{D}_{\mathrm{dg}}}

\DeclareMathOperator{\RHom}{\mathbb R\!\Hom}
\DeclareMathOperator{\lotimes}{\overset{\mathbb L}{\otimes}}
\newcommand{\cat}{\mathscr}
\newcommand{\opp}[1]{{#1}^{\mathrm{op}}}
\newcommand{\kat}{\mathsf}
\newcommand{\smallcat}{\mathfrak}

\newcommand{\Hqe}{\kat{Hqe}}

\newcommand{\basering}[1]{\mathbf{#1}}

\newtheorem{theorem}{Theorem}[subsection]
\newtheorem*{theorem*}{Theorem}
\newtheorem{proposition}[theorem]{Proposition}
\newtheorem{corollary}[theorem]{Corollary}

\newtheorem{lemma}[theorem]{Lemma}

\theoremstyle{remark}
\newtheorem{remark}[theorem]{Remark}
\newtheorem{example}[theorem]{Example}
\newtheorem{notation}[theorem]{Notation}

\theoremstyle{definition}
\newtheorem{definition}[theorem]{Definition}

\newtheorem{setup}[theorem]{Setup}

\AtAppendix{\counterwithin{theorem}{section}}

\numberwithin{equation}{section}
\title[Derived Gabriel-Popescu via derived injectives]{A derived Gabriel-Popescu theorem for t-structures via derived injectives}

\author{Francesco Genovese} 
\address[Francesco Genovese]{Universiteit Antwerpen, Departement Wiskunde-Informatica, Middelheimcampus,
Middelheimlaan 1,
2020 Antwerp, Belgium}
\email{francesco.genovese@uantwerpen.be}

\author{Julia Ramos Gonz\'alez}
\address[Julia Ramos Gonz\'alez]{Universiteit Antwerpen, Departement Wiskunde-Informatica, Middelheimcampus,
	Middelheimlaan 1,
	2020 Antwerp, Belgium}
\email{julia.ramosgonzalez@uantwerpen.be}

\thanks{The first named author acknowledges the support of the European Research Council (ERC) under Grant No. 817762 and the support of the Research Foundation - Flanders (FWO) under Grant No. G.0D86.16N.
The second named author is a postdoctoral fellow of the Research Foundation -  Flanders (FWO)}

\begin{document}
\begin{abstract}
    We prove a derived version of the Gabriel-Popescu theorem in the framework of dg-categories and t-structures. This exhibits any pretriangulated dg-category with a suitable t-structure (such that its heart is a Grothendieck abelian category) as a t-exact localization of a derived dg-category of dg-modules. We give an original proof based on a generalization of Mitchell's argument in \emph{A quick proof of the Gabriel-Popesco theorem} and involving derived injective objects. As an application, we also give a short proof that derived categories of Grothendieck abelian categories have a unique dg-enhancement.
\end{abstract}
	\maketitle

	\tableofcontents
	
	\addtocontents{toc}{\protect\setcounter{tocdepth}{1}}
	\section{Introduction}

	\emph{Grothendieck categories} appear naturally in many algebraic and geometric contexts. For instance, they play an essential role in noncommutative algebraic geometry: inspired by the Gabriel-Rosenberg reconstruction theorem \cite{gabriel-abeliancategories} \cite{rosenberg-spectraspaces} \cite{brandenburg-reconstruction}, Grothendieck categories are interpreted as (the categorial counterpart of) noncommutative spaces. The \emph{Gabriel-Popescu theorem} \cite{gabriel-popescu-original} characterizes Grothendieck abelian categories as the localizations of the module categories. More concretely, we can recover any Grothendieck category $\mathfrak G$ as a localizing subcategory of $\Mod(\smallcat a)$, where $\smallcat a$ is a small full subcategory of $\mathfrak G$ generating $\mathfrak G$.  As a consequence of this theorem, working with Grothendieck categories gets reduced to working with concrete categories fully determined by small data, namely, a set of generators.
	
	Moreover, the Gabriel-Popescu theorem draws a strong link between Grothendieck categories and Grothendieck topoi. Indeed, as observed in \cite{lowen-gabrielpopescu}, the Gabriel-Popescu theorem is nothing but the linear counterpart of Giraud's theorem \cite{giraud-analysissitus}; therefore, we can interpret Grothendieck categories as a linear version of Grothendieck topoi. More concretely, we can envision them as categories of sheaves on linear Grothendieck sites (see \cite{lowen-lineartopologies},  \cite{borceux-quinteiro-enrichedsheaves}), and make use, after suitable linearization, of the classical machinery available for sheaf topoi.
	
	\emph{Triangulated categories} also appear naturally in several algebraic and geometric contexts, with derived categories of abelian categories as a prime example. If abelian categories are to be noncommutative spaces, then derived categories (and more general triangulated categories) seem to be candidates for \emph{derived} noncommutative spaces. In this setting, Porta proved in \cite{porta-gabrielpopescu-triangulated} a triangulated version of the Gabriel-Popescu theorem which identifies well generated algebraic triangulated categories as the triangulated ``Grothendieck'' counterparts. In particular, the role of the categories of modules is now played by the derived categories of small dg-categories.
	
	It is well-known that triangulated categories are, however, not well-behaved enough in order to perform many relevant constructions (see, for example, \cite[\S1.1]{toen-lectures}). This can be overcome by the introduction of suitable higher-categorical enhancements, such as the aforementioned dg-categories, $A_\infty$-categories or $\infty$-categories. In this paper we aim to explore a suitable version of the Gabriel-Popescu theorem in an enhanced framework. We shall use \emph{pretriangulated dg-categories} (cf. \cite{bondal-kapranov-enhanced}) as our enhancements: they are indeed the most natural enhancements for derived categories and other triangulated categories of geometric interest. Moreover, they are naturally linear over any chosen commutative ring (or even dg-ring).

    We would like to think of a pretriangulated dg-category as a ``derived counterpart'' of an abelian category, hence an actual derived noncommutative space, but this is only a part of the story. Indeed, the world of triangulated categories (and their dg-models) does not seem \emph{per se} to communicate well with the ``classical (abelian) world''. The missing link we need is given by \emph{t-structures}. First introduced in \cite{beilinson-bernstein-deligne-perverse}, a t-structure on a triangulated category $\cat T$ is essentially the specification of an abelian subcategory $\cat T^\heartsuit$ (called the \emph{heart}) and ``cohomology functors'' $H^n \colon \cat T \to \cat T^\heartsuit$. Most triangulated categories of interest can be endowed with t-structures: a typical example is the derived categort $\dercomp(\mathfrak A)$ of an abelian category $\mathfrak A$, in which case $\dercomp(\mathfrak A)^\heartsuit \cong \mathfrak A$ and the cohomology functors are the usual ones. T-structures on pretriangulated dg-categories are simply understood as t-structures on their underlying triangulated categories.
    
    By endowing pretriangulated dg-categories with t-structures, we are really able to generalize results from the abelian to the derived framework. This can be seen as the actual shift from noncommutative spaces to derived noncommutative spaces. An important example of this program is the notion of \emph{derived injective}. Derived injectives are suitable derived versions of injective objects of abelian categories and they intrinsically depend on the presence of a t-structure. Just as objects in abelian categories (often) have injective resolutions, objects in pretriangulated dg-categories with t-structures (often) have ``derived injective resolutions'': these are extensively studied in \cite{genovese-lowen-vdb-dginj} and they will also be a crucial ingredient of the present work.
    
    The main goal of our paper is to generalize the Gabriel-Popescu theorem to this derived setting, where t-structures play the key role. We will need assumptions which make the given pretriangulated dg-category ``behave as a Grothendieck abelian category''; in particular, the heart of the t-structure will be Grothendieck. The prototypical example is given by the \emph{derived dg-category} $\dercompdg(\smallcat u)$ of a dg-category $\smallcat u$ which is concentrated in nonpositive degrees. Mirroring the classical Gabriel-Popescu theorem, we will describe any pretriangulated dg-category with such well-behaved t-structure as a t-exact localization of a dg-category of the form $\dercompdg(\smallcat u)$. 
    
    We remark that a similar endeavour has been already undertaken by Lurie \cite[\S C]{lurie-SAG} in the language of $\infty$-categories. In particular, he introduces \emph{(Grothendieck) pre-stable $\infty$-categories} as the key objects which generalize (Grothendieck) abelian categories and he proves an $\infty$-categorical Gabriel-Popescu theorem in that language \cite[Theorem C.2.1.6]{lurie-SAG}; this also yields a variant for presentable stable $\infty$-categories endowed with suitable t-structures \cite[Corollary C.2.1.10]{lurie-SAG}. Lurie's proofs and techniques could be quite directly translated to our language of dg-categories (with the appropriate changes), but we take a different approach. In \cite{mitchell-gabrielpopescu} Mitchell gives a very concise proof of the classical Gabriel-Popescu theorem where injective objects play a key role. We generalize that proof to our derived setup, and the aforementioned derived injectives will be pivotal.  
    
    We would like to point out that while the classical and the triangulated Gabriel-Popescu theorems provide a full characterization of the respective notions of ``Grothendieck categories'', our result only shows that the chosen family of ``Grothendieck categories'' can be written as a localization, but not that all localizations do necessarily belong to this family. The full characterization of the t-exact localizations of dg-categories of the form $\dercompdg(\smallcat u)$ will be addressed in future work.
    
    This \emph{derived Gabriel-Popescu theorem} will also be instrumental in an upcoming project where we attempt to generalize tensor products of Grothendieck abelian categories and well generated dg-categories (cf. \cite{lowen-julia-boris-tensorproduct} and \cite{lowen-julia-tensordg-wellgen}) to the setting of t-structures. In particular, inspired by the abelian framework and the techniques used in \cite{lowen-julia-boris-tensorproduct}, the definition of a suitable notion of ``t-dg-Grothendieck site'' is work in progress. 
	
	\subsection*{Structure of the paper} In \S\ref{sec:dgcategories-tstructures} we present the background and preliminary results that will be necessary for the rest of the paper. We first provide a concise survey on the basic concepts and results of the theory of dg-categories and quasi-functors. This is contained in \S\ref{subsec:dgcategories-quasifunctors}. We then review in \S\ref{sect:tstructures} the definition and essential properties of t-structures. More interestingly, we introduce the notion of \emph{t-monomorphism} and \emph{t-epimophism} and provide a ``t-epi/t-mono'' factorization in triangulated categories with a t-structure (cf. \S\ref{subsubsec:epimono}) which will be essential in the technical core of the paper. We close this preliminary section with a survey on derived injectives, contained in \S\ref{sect:dginj}. In particular, the relation between preservation of derived injectives and t-exactness (cf. \S\ref{subsubsec:texactness-dginj}) will be a key tool for our purposes. While the basic material on dg-categories and t-structures can be safely avoided by the expert reader and consulted afterwards if necessary, the main results regarding the ``t-epi/t-mono'' factorizations (cf. \Cref{prop:epimono_leftaisle}) and the relation between t-exactness and preservation of derived injectives (cf. \Cref{prop:rightadj_preserve_inj_leftadj_exact_abelian}) play an important role in the rest of the paper and are worth a closer look.
	
	As mentioned above, the main goal of this article is to provide a derived version of the Gabriel-Popescu theorem for a family of \emph{t-dg-categories} (pretriangulated dg-categories endowed with a t-structure) capturing the essence of being ``Grothendieck'' in this setting (in particular, the assumptions guarantee that the heart of the t-structure is a Grothendieck category). We determine this family in \S\ref{setup:GabrielPopescu} and we show that it contains the main geometrically-relevant instances of t-dg-categories, such as the enhancements of derived categories of Grothendieck categories or the dg-derived categories of small dg-categories concentrated in nonpositive degrees, in both cases endowed with their canonical t-structures (cf. \Cref{example:derivedcategory} and \Cref{example:dercat_Grothendieckabelian}). We then show that this family is subject to a derived Gabriel-Popescu theorem, which is the main result of this paper:
	
	\begin{theorem*}[\Cref{thm:GabrielPopescu}]
		Let $\cat A$ be a t-dg-category in our family (cf. \ref{setup:GabrielPopescu}). Let $\cat U$ denote the full dg-subcategory of generators of the t-structure on $\cat A$. Set $\smallcat u = \tau_{\leq 0}\cat U$ and $j:\smallcat u \to \cat A$ the natural morphism. Then, the quasi-functor
		\begin{equation*} 
			\begin{split}
				G \colon \cat A & \to \dercompdg(\smallcat u), \\
				A & \mapsto \cat A(j(-), A)
			\end{split}
		\end{equation*}
	is quasi-fully faithful and has a t-exact left adjoint quasi-functor
		\begin{equation*} 
				F \colon \dercompdg(\smallcat u) \to \cat A
		\end{equation*}
	such that $H^0(F)(\smallcat u(-,U)) \cong U$ in $H^0(\cat A)$, naturally in $U \in H^0(\smallcat u)$.
	\end{theorem*}
	In particular, we show that one can deduce the classical Gabriel-Popescu theorem as a direct consequence of our main theorem (cf. \Cref{coroll:classicalGP}).
	
	The proof of the theorem is provided in \S\ref{sect:Gabriel-Popescu}. Given its complexity, we break it down in different steps. The existence of the left adjoint and its right t-exactness can be proven more or less directly (cf. \Cref{prop:existence-left-adjoint}, \Cref{prop:right-texact}). Our approach to the proof of the left t-exactness of $F$ follows a derived version of the proof of the classical Gabriel-Popescu theorem by Mitchell \cite{mitchell-gabrielpopescu}. More concretely, by providing a derived version of Mitchell's lemma (cf. \Cref{lemma:Mitchell_extension}), we are able to show in \Cref{prop:preserve-dginj} that $G$ preserves derived injectives. This is enough to obtain the left t-exactness by making use of our preliminary study on derived injectives and t-exactness  (cf. \Cref{prop:rightadj_preserve_inj_leftadj_exact_abelian}). In addition, the derived version of Mitchell's lemma also allows us to show the quasi-fully faithfulness of $G$ (cf. \cref{prop:rightadj-fully-faithful}), which concludes the proof.
	
	As a direct application of our result, we provide in \S\ref{subsec:application} an alternative proof of the uniqueness of dg-enhancements for derived categories of Grothendieck categories.
	
    During our initial investigations on derived injectives, we proved a Baer criterion for derived injectives on the derived category of a dg-algebra. Although in the end this result was not necessary for our proof, we believe it to be of independent interest and have therefore included it in an appendix (cf. \Cref{appendix:dginj}).
    
	\subsection*{\emph{Acknowledgements}} 
	The authors would like to thank Wendy Lowen and Michel Van den Bergh for interesting discussions. The first named author also thanks Alberto Canonaco and Paolo Stellari for interesting discussions around the application on uniqueness of dg-enhancements (see \S \ref{subsec:application}).
	
	\addtocontents{toc}{\protect\setcounter{tocdepth}{2}}
	\section{Dg-categories and t-structures}\label{sec:dgcategories-tstructures}
	We fix once and for all a base \emph{commutative dg-ring} $\basering k$, which we also assume to be strictly concentrated in nonpositive degrees. Unless otherwise specified our constructions are in the $\basering k$-linear context,\footnote{This will be the $H^0(\basering k)$-linear context for ordinary (and also triangulated) categories and the $H^*(\basering k)$-linear context for graded categories.} although we do not always say it. 
	
	We fix a Grothendieck universe $\mathfrak U$. From this point on, all categories are $\mathfrak U$-categories, small categories are $\mathfrak U$-small categories and small limits and colimits are $\mathfrak U$-small limits and colimits. We additionally fix two extra universes $\mathfrak U \in \mathfrak V$ and $\mathfrak V \in \mathfrak W$, such that for every $\mathfrak U$-small (dg) category its category of (dg) modules is $\mathfrak V$-small and for every $\mathfrak V$-small (dg) category its category of (dg) modules is $\mathfrak W$-small. From this point on, we refrain ourselves to indicate when an enlargement of the universe is required (from $\mathfrak U$ to $\mathfrak V$, or from $\mathfrak V$ to $\mathfrak W$), as this will be evident from the context.  
	\subsection{Dg-categories and quasi-functors}\label{subsec:dgcategories-quasifunctors}
	We shall work with $\basering k$-linear dg-categories, namely, categories enriched over dg-$\basering k$-modules. Their homotopy categories (resp. graded homotopy categories) will carry a natural $H^0(\basering k)$-linear (resp. $H^*(\basering k)$-linear) structure.  We refer to \cite[\S 2]{lunts-schnurer-smoothness-equivariant} for a specific account on $\basering k$-linear dg-categories and some of their homotopy-theoretical features; practically, everything works as in the case where $\basering k$ is an ordinary commutative ring. We also refer to \cite{keller-dgcat} for a general survey. Here, we recollect some notation and results we shall need throughout the paper.
	\subsubsection{Dg-modules} We denote by $\compdg(\basering k)$ the dg-category of $\basering k$-dg-modules. Let $\cat A$ be a small a dg-category. We define the dg-category of \emph{(right) $\cat A$-dg-modules} as:
	\[
	\compdg(\cat A) = \Fun_{\basering k}(\opp{\cat A}, \compdg(\basering k)),
	\]
    namely the dg-category of dg-functors $\opp{\cat A} \to \compdg(\basering k)$.
	
	The \emph{derived category $\dercomp(\cat A)$} is defined as the Verdier quotient of $H^0(\compdg(\cat A))$ by the acyclic dg-modules. A dg-model of $\dercomp(\cat A)$ is given by the dg-category $\hproj(\cat A)$ of \emph{h-projective} dg-modules, namely dg-modules $M \in \compdg(\cat A)$ such that $\compdg(\cat A)(M,-)$ preserves acyclic objects. Thus, we have an ($H^0(\basering k)$-linear) equivalence:
	\[
	H^0(\hproj(\cat A)) \cong \dercomp(\cat A)
	\]
	which we will sometimes view as an identification, implicitly taking suitable \emph{h-projective resolutions}, that is, quasi-isomorphisms of the form
	\[
	Q(M) \to M,
	\]
	for $M \in \compdg(\cat A)$, where $Q(M)$ is h-projective. We also sometimes use the following notation:
	\begin{equation} \label{eq:dercompdg_hproj}
	\dercompdg(\cat A) = \hproj(\cat A),
	\end{equation}
	interpreting $\hproj(\cat A)$ as the \emph{derived dg-category} of $\cat A$.
	
	Notice that for any $A \in \cat A$, the representable dg-module $\cat A(-, A)$ is h-projective by the dg-Yoneda lemma. So, the Yoneda embedding gives rise to a fully faithful dg-functor
	\[
	h_{\cat A} \colon \cat A \hookrightarrow \hproj(\cat A),
	\]
    which in turn induces the so-called \emph{derived Yoneda embedding}:
    \[
    H^0(\cat A) \hookrightarrow \dercomp(\cat A).
    \]

    \subsubsection{Pretriangulated dg-categories} \label{subsubsec:dgcat_pretr}
    Let $\cat A$ be a dg-category. We denote by $\pretr(\cat A)$ the smallest full dg-subcategory of $\compdg(\cat A)$ which contains all the representables $\cat A(-,A)$ and it is closed under taking shifts and mapping cones. We remark that the dg-category $\hproj(\cat A)$ is indeed closed under shifts and mapping cones and contains the representables, so the Yoneda embedding factors as follows:
    \[
    h_{\cat A} \colon \cat A \hookrightarrow \pretr(\cat A) \hookrightarrow \hproj(\cat A).
    \]
    We say that $\cat A$ is \emph{pretriangulated} if $\cat A \hookrightarrow \pretr(\cat A)$ is a quasi-equivalence; we say that it is \emph{strongly pretriangulated} if $\cat A \hookrightarrow \pretr(\cat A)$ is a ``strict'' dg-equivalence.  We refer the reader to the seminal work \cite{bondal-kapranov-enhanced} for the essential features of the theory of pretriangulated categories. 
    
    The dg-categories $\compdg(\cat A), \hproj(\cat A)$ and $\pretr(\cat A)$ are all strongly pretriangulated. In particular, if $\cat A$ is pretriangulated, we can replace it up to quasi-equivalence with $\pretr(\cat A)$, which is strongly pretriangulated.
    
    The crucial property of a pretriangulated dg-category $\cat A$ is that it has \emph{functorial shifts and cones} (of closed degree $0$ morphisms). Hence, with the obvious choice of shifts and distinguished triangles, the homotopy category $H^0(\cat A)$ becomes a triangulated category. If $\cat A$ is moreover strongly pretriangulated, its functorial shifts and cones exist up to ``strict'' isomorphism in $\cat A$. We now describe a variant of the key property of such functorial cones, in a specific way which will be suitable for our purposes.
    \begin{lemma} \label{lemma:functorialcones_strict}
        Let $\cat A$ be a strongly pretriangulated dg-category, and let
        \[
        \begin{tikzcd}
        A \arrow[r, "f"] \arrow[d, "u"] & B \arrow[d, "v"] \\
        A' \arrow[r, "f'"]              & B'              
        \end{tikzcd}
        \]
        be a \emph{strictly} commutative diagram of closed degree $0$ morphisms in $\cat A$. Let $\cone(f)$ and $\cone(f')$ be the functorial cones in $\cat A$ of respectively $f$ and $f'$. Then, there is a \emph{unique} closed degree $0$ morphism
        \[
        w \colon \cone(f) \to \cone(f')
        \]
        such that the following diagram (of closed degree $0$ morphisms) is strictly commutative:
        \[
        \begin{tikzcd}
        {\cone(f)[-1]} \arrow[r] \arrow[d, "{w[-1]}", dotted] & A \arrow[r, "f"] \arrow[d, "u"] & B \arrow[d, "v"] \arrow[r] & \cone(f) \arrow[d, "w", dotted] \\
        {\cone(f')[-1]} \arrow[r]                             & A' \arrow[r, "f'"]              & B' \arrow[r]               & \cone(f')                      
        \end{tikzcd}
        \]
        The rows of the above diagram induce distinguished triangles in $H^0(\cat A)$.
    \end{lemma}

    \subsubsection{Resolutions of dg-categories} For any dg-category $\cat A$, there is a dg-functor $Q(\cat A) \to \cat A$ with the following properties:
    \begin{itemize}
        \item It is a \emph{quasi-equivalence}, namely, it induces an equivalence of graded $H^*(\basering k)$-linear categories
        \[
        H^*(Q(\cat A)) \xrightarrow{\sim} H^*(\cat A).
        \]
        \item $Q(\cat A)$ is \emph{h-projective}, namely, it has h-projective hom complexes. This implies that it is also \emph{h-flat}, namely, $- \otimes_{\basering k} \cat A(A,B)$ preserves acyclic $\basering k$-dg-modules for all $A,B \in \Ob(\cat A)$.
    \end{itemize}
    This can be found, for example in \cite[\S3.5]{drinfeld-dgquotients}.
    $Q(\cat A) \to \cat A$ is called an \emph{h-projective (and hence h-flat) resolution} of $\cat A$. We can use it to define a \emph{derived tensor product} of dg-categories:
    \[
    \cat A \lotimes_{\basering k} \cat B = Q(\cat A) \otimes_{\basering k} \cat B \cong \cat A \otimes_{\basering k} Q(\cat B).
    \]
    The above isomorphisms are taken in the \emph{homotopy category of dg-categories} $\Hqe(\basering k)$, namely, the localization of the category of (small) dg-categories $\kat{dgCat}(\basering k)$ along quasi-equivalences: the derived tensor product is well defined there up to isomorphism. For an account on $\Hqe(\basering)$ we refer the reader to \cite{toen-dgcat-invmath}.
    
    \subsubsection{Dg-bimodules and quasi-functors} \label{subsubsec:bimod_qfun}
    Let $\cat A$ and $\cat B$ be dg-categories. An \emph{$\cat A$-$\cat B$-dg-bimodule} is by definition a dg-functor
    \[
    F \colon \opp{\cat B} \otimes_{\basering k} \cat A \to \compdg(\basering k),
    \]
    namely, an element of the dg-category $\compdg(\cat B \otimes \opp{\cat A})$. We shall often write
    \[
    F_A^B = F(B,A),
    \]
    with the convention ``contravariant-above, covariant-below''. Sometimes, we shall also denote by $\cat A$ or $h_{\cat A}$ or even $h$ the diagonal bimodule of a dg-category $\cat A$, namely:
    \[
    h_A^{A'} = \cat A_A^{A'} = \cat A(A',A).
    \]
    
    Now, assume that $\cat A$ or $\cat B$ is h-flat, if necessary taking an h-flat resolution, so that
    \[
    \cat B \lotimes_{\basering k} \opp{\cat A} = \cat B \otimes_{\basering k} \opp{\cat A}.
    \]
    An $\cat A$-$\cat B$-dg-bimodule $F \in \compdg(\cat B \otimes_{\basering k} \opp{\cat A})$ is called a \emph{quasi-functor} if it is \emph{right quasi-representable}, namely, for any $A \in \cat A$ there is an object $\Phi_F(A) \in \cat B$ and an isomorphism
    \[
    \cat B(-,\Phi_F(A)) \xrightarrow{\sim} F_A.
    \]
    We shall denote such quasi-functor as an arrow
    \[
    F \colon \cat A \to \cat B.
    \]
    Moreover, $F$ induces a graded $H^*(\basering k)$-linear functor
    \begin{align*}
        H^*(F) \colon H^*(\cat A) & \to H^*(\cat B), \\
        A & \mapsto \Phi_F(A).
    \end{align*}
    Analogously, it induces an $H^0(\basering k)$-linear functor
    \[
    H^0(F) \colon H^0(\cat A) \to H^0(\cat B).
    \]
    We say that $F$ is \emph{quasi-fully faithful} if $H^*(F)$ is a fully faithful graded functor. We say that it is \emph{quasi-essentially surjective} if $H^0(F)$ is essentially surjective. Finally, we say that it is a \emph{quasi-equivalence} if it is both quasi-fully faithful and quasi-essentially surjective, namely, if $H^*(F)$ is an equivalence of graded categories. Using the results in \cite{genovese-adjunctions} (in particular, the derived duality of bimodules), we can characterize quasi-equivalences as quasi-functors which admit suitable inverse quasi-functors.
    
    We shall denote by $\RHom(\cat A,\cat B)$ the full dg-subcategory of $\hproj(\cat B \lotimes_{\basering k} \opp{\cat A})$ spanned by the quasi-functors. This gives the internal hom of the symmetric monoidal category $(\Hqe(\basering k), \lotimes_{\basering k})$ (cf. \cite{toen-dgcat-invmath}, \cite{canonaco-stellari-internalhoms}). Sometimes, we shall abuse notation and identify quasi-functors with the objects of $\RHom(\cat A, \cat B)$, implicitly taking h-projective resolutions.
    \subsubsection{Adjoint quasi-functors} \label{subsubsec:qfun_adj}
    Let $\cat A$ and $\cat B$ be dg-categories, and let
    \[
    F \colon \cat A \rightleftarrows \cat B : G
    \]
    be quasi-functors. We say that \emph{$F$ is left adjoint to $G$}, in symbols $F \dashv G$, if the induced graded functors
    \[
    H^*(F) \colon H^*(\cat A) \rightleftarrows H^*(\cat B) : H^*(G)
    \]
    are part of an adjunction $H^*(F) \dashv H^*(G)$. Adjunctions of quasi-functors are reasonably well-behaved, in the sense that the existence of an adjoint at the graded level implies that there is actually an adjoint at the level of quasi-functors:
    \begin{lemma}[see also {\cite[Remark 3.9]{lowen-julia-tensordg-wellgen}}] \label{lemma:adjoints_quasifunctors_cohomology}
    Let $G \colon \cat B \to \cat A$ be a quasi-functor. Assume that $H^*(G) \colon H^*(\cat B) \to H^*(\cat A)$ has a left adjoint $L \colon H^*(\cat A) \to H^*(\cat B)$ as a graded functor. Then, $G$ has a left adjoint $F$  as a quasi-functor. Moreover, $F$ satisfies $H^*(F) \cong L$.
    
    If $\cat A$ and $\cat B$ are pretriangulated dg-categories, the above claim holds if we replace $H^*$ with $H^0$ everywhere.
    \end{lemma}
    \begin{proof}
        $G \colon \cat B \to \cat A$ is a quasi-functor, hence for any $B \in \cat B$ we have an isomorphism
        \[
        \cat A(-,\Phi_G(B)) \xrightarrow{\sim} G_B
        \]
        in $\dercomp(\cat A)$. The mapping $B \mapsto \Phi_G(B)$ defines the graded functor $H^*(G)$. For $A \in \cat A$, we have by assumption an isomorphism
        \begin{equation} \label{eq:adjoints_quasifunctors_cohomology} \tag{$\ast$}
        H^*(\cat B)(L(A), -) \cong H^*(\cat A)(A, \Phi_G(-)) \cong H^*(G^A)
        \end{equation}
        of graded left $H^*(\cat B)$-modules. Thanks to the (graded) Yoneda lemma, this yields an element
        \[
        e \in H^0(G^A_{L(A)}).
        \]
        Then, applying the \emph{derived} Yoneda lemma, we deduce that $e$ yields an isomorphism in $\dercomp(\opp{\cat B})$, for all $A \in \cat A$:
        \[
        \cat B(L(A),-) \xrightarrow{\sim} G^A,
        \]
        which induces the above \eqref{eq:adjoints_quasifunctors_cohomology} when taking $H^*$. This means that $G$ is \emph{left quasi-representable}, hence by \cite[Proposition 7.1]{genovese-adjunctions} we know that $G$ has a left adjoint quasi-functor $F$. By construction, $H^*(F)$ is precisely $L$.
        
        The second part of the claim is immediate, since by assumption we are allowed to take shifts of objects in $\cat A$.
    \end{proof}
    
    \subsubsection{Induction and restriction quasi-functors} \label{subsubsec:ind_res}
    
    Let $\cat A$ and $\cat B$ be dg-categories, and let $F \colon \cat A \to \cat B$ be a dg-functor. Then, we have an \emph{induction dg-functor}
    \begin{equation*}
    \Ind_F \colon \compdg(\cat A) \to \compdg(\cat B), \\
    \end{equation*}
    which is left adjoint to the \emph{restriction dg-functor}
    \begin{equation*} 
    \begin{split}
        \Res_F \colon \compdg(\cat B) & \to \compdg(\cat A), \\
        M & \mapsto M \circ F.
    \end{split}
    \end{equation*}
    We refer to \cite{drinfeld-dgquotients} for the definition of $\Ind_F$ and we recall from there some of its relevant properties:
    \begin{itemize}
        \item $\Ind_F$ preserves representable dg-modules, namely, there is an isomorphism in $\compdg(\cat B)$:
        \[
        \Ind_F(\cat A(-,A)) \cong \cat B(-,F(A)),
        \]
        natural in $A \in \cat A$.
        \item $\Ind_F$ preserves h-projective modules and hence induces a dg-functor
        \begin{equation}
            \Ind_F \colon \dercompdg(\cat A) \to \dercompdg(\cat B),
        \end{equation}
        recalling notation \eqref{eq:dercompdg_hproj}.
        
        \item If $F$ is fully faithful, the same is true for $\Ind_F$; if $F$ is a quasi-equivalence, the same is true for $\Ind_F$. 
        
        \item If $F$ is a dg-equivalence with inverse $G$, then
        \begin{equation}
            \Ind_F \colon \dercompdg(\cat A) \rightleftarrows \dercompdg(\cat B) : \Ind_G
        \end{equation}
        are dg-equivalences, inverse to each other.
    \end{itemize}
    
    The restriction dg-functor $\Res_F$ induces a \emph{restriction quasi-functor}
    \begin{equation} \label{eq:restriction_quasifunctor}
        \Res_F \colon \dercompdg(\cat B) \to \dercompdg(\cat A).
    \end{equation}
    It is defined as the following dg-bimodule:
    \[
    (\Res_F)^N_M = \compdg(\cat A)(N, M \circ F),
    \]
    for $M \in \hproj(\cat B)$ and $N \in \hproj(\cat A)$. Since $N$ is h-projective, we have a quasi-isomorphism
    \[
    \hproj(\cat A)(N, Q(M \circ F)) \xrightarrow{\sim} \compdg(\cat A)(N, M \circ F) = (\Res_F)_M^N,
    \]
    natural in $N$, where $Q(M \circ F) \to M \circ F$ is an h-projective resolution. This proves that $\Res_F$ is indeed a quasi-functor $\hproj(\cat B) \to \hproj(\cat A)$.
    
    Moreover, it is clear that there is an adjunction of quasi-functors (cf. \S \ref{subsubsec:qfun_adj}):
    \begin{equation} \label{eq:indres_qfun_adj}
        \Ind_F \colon \dercompdg(\cat A) \rightleftarrows \dercompdg(\cat B) : \Res_F.
    \end{equation}
    \subsubsection{Dg-quotients and localizations} \label{subsubsec:dgquotients_localizations}
    The theory of dg-quotients was initiated by Keller in \cite{keller-cyclic} and further developed by Drinfeld in \cite{drinfeld-dgquotients}. Given a dg category $\cat A$ and a full dg-subcategory $\cat A' \subseteq \cat A$, the \emph{dg-quotient} $\cat A/ \cat A'$ is a dg-category satisfying the following universal property in $\Hqe(\basering k)$ (cf. \cite{tabuada-dgquotient}):
    \[
    \RHom(\cat A/\cat A',\cat B) \cong \RHom_{\cat A'}(\cat A,\cat B)
    \]
    for all $\cat B$ in $\Hqe(\basering k)$, where $\RHom_{\cat A'}(\cat A,\cat B)$ denotes the full dg-subcategory of $\RHom(\cat A,\cat B)$ given by the quasi-functors $F$ such that $H^0(F):H^0(\cat A) \to H^0(\cat B)$ annihilates the objects of $\cat A'$ (seen as objects in $H^0(\cat A)$). Given $\cat A' \subseteq \cat A$ as above, the dg-quotient $\cat A/\cat A'$ does exist. We refer the reader to \cite{drinfeld-dgquotients} for an explicit construction.
    
    In particular, the dg-quotient provides an enhancement of the Verdier quotient of triangulated categories. More concretely, we have the following:
    
    \begin{theorem}[{\cite[Thm 3.4]{drinfeld-dgquotients}}]
    Given a strongly pretriangulated dg-category $\cat A$ and a dg-subcategory $\cat A'$, we have that
    \[
    H^0(\cat A) / H^0(\cat A') \cong H^0(\cat A /\cat A').
    \]
    \end{theorem}
    
    In the localization theory of triangulated categories, Bousfield localizations play an important role. A Bousfield localization of a triangulated category $\cat T$ is an exact functor $L: \cat T \to \cat T$ that can be realised as a composition
    \[
    \cat T \to \cat S \to \cat T
    \]
    where $\cat S$ is a replete subcategory of $\cat T$, the second functor is fully faithful and the first functor is the left adjoint of the second one \cite[Prop. 2.6.1]{krause-localizationtheory}. In this situation, $\cat S$ is actually equivalent to the Verdier quotient $\cat T/\ker(L)$ and the first functor is just the canonical quotient functor (see \cite[Prop 4.9.1]{krause-localizationtheory}). In particular, $\Img(L) \cong \cat S \cong \cat T/\cat S$.
    
    When we restrict to the case of a well generated triangulated category $\cat T$ in the sense of Neeman (cf. \cite[Def 8.1.4]{neeman-triangulated}), Verdier quotients and Bousfield localizations are equivalent approaches to the localization theory of $\cat T$. More explicitly, Bousfield localizations of $\cat T$ with kernel generated by a set of objects are in bijection with localizing subcategories of $\cat T$ generated by a set of objects \cite[Thm 7.2.1 \& Prop 5.2.1]{krause-localizationtheory}. One direction is provided by taking the composition of the Verdier quotient with its left adjoint (which exists because localizing subcategories are closed under small coproducts and the fact that Brown representability holds for well generated triangulated categories \cite[Thm 5.1.1]{krause-localizationtheory}). The other is given by taking the kernel of the localization functor, which is localizing. The generation by a set has to be imposed in order for the Verdier quotient to be well generated as well. 
    
    Bousfield localizations, as Verdier quotients, can also be enhanced to the dg-realm (see \cite[\S3.2]{lowen-julia-tensordg-wellgen}) as follows:
    
    \begin{definition}\label{def:dg_Bousfield_localization}
	Let $\cat A, \cat B$ be pretriangulated dg categories and $i: \cat B \to \cat A$ a quasi-fully faithful dg functor. We say that $i \in \RHom(\cat B, \cat A)$ is a \emph{dg Bousfield localization} of $\cat A$ if $H^0(i): H^0(\cat B) \hookrightarrow H^0(\cat A)$ admits a left adjoint. In view of \Cref{lemma:adjoints_quasifunctors_cohomology}, this is equivalent to saying that $i: \cat B \to \cat A$ has a left adjoint quasi-functor $G \in \RHom(\cat A, \cat B)$.
    \end{definition}

    \begin{remark}
    Observe that a dg-Bousfield localization induces a classical Bousfield localization of the corresponding underlying triangulated category $H^0(\cat A)$.
    \end{remark}
    \begin{remark}\label{remqlpretriang}
    Given a dg-Bousfield localization $i: \cat B \to \cat A$ with left adjoint quasi-functor $G: \cat A \to \cat B$, one has that $\cat B$ is isomorphic in $\Hqe(\basering k)$ to the dg-quotient $\cat A/ \ker(H^0(G))$ where we see $\ker(H^0(G))$ as a full dg-subcategory of $\cat A$ via the natural enhancement $\cat A$ of $H^0(\cat A)$.
    \end{remark}

    In parallel fashion to what happened in the triangulated framework, when we restrict ourselves to the realm of pretriangulated well generated dg-categories (i.e. pretriangulated dg categories whose underlying triangulated category is well generated), there is an order-reversing bijective correspondence between localizing subcategories and Bousfield localizations \cite[\S 3.2.3]{lowen-julia-tensordg-wellgen}. Namely, for a well generated dg category $\cat A$, there is a bijection between:
    \begin{itemize}
    	\item The set of localizing subcategories of $H^0(\cat A)$ generated by a set;
    	\item The set of dg-Bousfield localizations of $\cat A$ with the kernel of the left adjoint (at the $H^0$-level) generated by a set.
    \end{itemize} 

    \subsubsection{Truncations of dg-categories} \label{subsubsec:truncations_dgcat}
    Let $\cat A$ be a dg-category. We define a dg-category $\tau^{\leq 0} \cat A$ as follows: $\Ob \tau^{\leq 0} \cat A = \Ob \cat A$, and for any $A,B \in \Ob \cat A$ we set
    \begin{equation*}
        (\tau^{\leq 0} \cat A)(A,B) = \tau_{\leq 0} (\cat A(A,B)),
    \end{equation*}
    where $\tau_{\leq 0}$ is the (smart) truncation of $\basering k$-dg-modules. More explicitly:
    \begin{equation*}
    \begin{cases}
    (\tau^{\leq 0} \cat A)(A,B)^n = \cat A(A,B)^n & \text{if $n < 0$}, \\
    (\tau^{\leq 0} \cat A)(A,B)^n = 0 & \text{if $n > 0$}, \\
    (\tau^{\leq 0} \cat A)(A,B)^0 = Z^0(\cat A(A,B)).
    \end{cases}
    \end{equation*}
    There is a natural dg-functor
    \begin{equation*}
    i^{\leq 0} \colon \tau^{\leq 0} \cat A \rightarrow \cat A,
    \end{equation*}
    which is the identity on objects and the inclusion map $\tau_{\leq 0} \cat A(A,B) \hookrightarrow \cat A(A,B)$ on hom-complexes.
    
	\subsection{t-structures on dg-categories}\label{sect:tstructures}
	t-structures on dg-categories are understood as t-structures on their homotopy categories.
    \begin{definition} \label{def:tstruct_dgcat}
    A \emph{t-structure} on a pretriangulated dg-category $\cat A$ is a t-structure on the homotopy category $H^0(\cat A)$ (see \cite{beilinson-bernstein-deligne-perverse}). 
    
    We shall denote by $\cat A_{\leq n}$ and $\cat A_{\geq n}$ the full
    dg-subcategories of $\cat A$ whose objects are the same as the
    aisles $H^0(\cat A)_{\leq n}$ and $H^0(\cat A)_{\geq n}$, so that
    \begin{align*}
        H^0(\cat A_{\leq n}) &= H^0(\cat A)_{\leq n}, \\
        H^0(\cat A_{\geq n}) &= H^0(\cat A)_{\geq n}.
    \end{align*}
    Such t-structure on $\cat A$ will often be denoted by the pair $(\cat A_{\leq 0}, \cat A_{\geq 0})$.
    
    A pretriangulated dg-category $\cat A$ with a t-structure $(\cat A_{\leq 0}, \cat A_{\geq 0})$ will be called a \emph{t-dg-category}.
    \end{definition}
    \begin{remark}
        Let $\cat A$ be a t-dg-category. We shall denote by
        \[
        H^0_t \colon H^0(\cat A) \to H^0(\cat A)^\heartsuit
        \]
        the cohomological functor associated to the t-structure. We also set
        \[
        H^n_t(-) = H^0_t(-[n]).
        \]
        We shall often ease notation and write $H^n$ instead of $H^n_t$.
    \end{remark}
    \begin{definition}
        Let $F \colon \cat A \to \cat B$ be a quasi-functor between t-dg-categories. We say that $F$ is \emph{left t-exact} (resp. \emph{right t-exact}) if $H^0(F)$ has this property, namely, it maps $H^0(\cat A_{\geq 0})$ to $H^0(\cat B_{\geq 0})$ (resp. it maps $H^0(\cat A_{\leq 0})$ to $H^0(\cat B_{\leq 0})$).
        
        We also say that $F$ is \emph{t-exact} if it is both left and right t-exact.
    \end{definition}
    A prototypical example of t-dg-categories is given by dg-derived categories of \emph{dg-categories concentrated in nonpositive degrees}.
    \begin{proposition} \label{prop:dercomp_naturaltstruct}
    Let $\smallcat a$ be a ($\basering k$-linear) dg-category concentrated in nonpositive degrees, namely, such that
    \begin{equation} \label{eq:dgcat_nonpositive}
    H^i(\smallcat a(A,B))=0, \qquad i>0,
    \end{equation}
    for all objects $A, B \in \smallcat a$. Then, the derived dg-category $\dercompdg(\smallcat a)=\hproj(\smallcat a)$ has a t-structure such that:
    \begin{align*}
        \dercompdg(\smallcat a)_{\leq 0} &= \{ M \in \dercompdg(\smallcat a) : H^i(M) = 0 \ \forall\, i > 0 \}, \\
        \dercompdg(\smallcat a)_{\geq 0} &= \{ M \in \dercompdg(\smallcat a) : H^i(M) = 0 \ \forall\, i < 0 \}.
    \end{align*}
    
    This t-structure will be referred to as \emph{natural} or \emph{canonical}. As a t-dg-category, $\dercompdg(\smallcat a)$ will be always endowed with this t-structure, unless otherwise specified. The heart $H^0(\dercompdg(\smallcat a))^\heartsuit = \dercomp(\smallcat a)^\heartsuit$ is identified with the abelian ($H^0(\basering k)$-linear) category $\Mod(H^0(\smallcat a))$ of right $H^0(\smallcat a)$-modules.
    \end{proposition}
    \begin{proof}
    This can be proved by adapting \cite[Lem 2.2, Prop 2.3]{amiot-cluster}.
    \end{proof}
	\begin{remark}
    Since our base dg-ring $\basering k$ is by assumption (strictly) concentrated in nonpositive degrees, the derived dg-category $\dercompdg(\basering k)$ has the natural t-structure described in the above Proposition \ref{prop:dercomp_naturaltstruct}.
    \end{remark}

	\subsubsection{Exact functors and t-structures}
	We now temporarily forget about dg-categories and discuss a few well-known results (cf. \cite[Proposition 1.3.17]{beilinson-bernstein-deligne-perverse}) which involve only triangulated categories endowed with t-structures and triangulated functors between them. Clearly, everything can be immediately generalized to t-dg-categories.
	\begin{proposition} \label{prop:induced_functor_heart}
		Let $\cat T$ and $\cat S$ be ($H^0(\basering k$)-linear) triangulated categories endowed with t-structures $(\cat T_{\leq 0}, \cat T_{\geq 0})$ and $(\cat S_{\leq 0}, \cat S_{\geq 0})$. Let $F \colon \cat T \to \cat S$ be a triangulated functor. There is an induced functor
		\begin{equation}
			\begin{split}
				F^\heartsuit & \colon \cat T^\heartsuit  \to \cat S^\heartsuit, \\
				A & \mapsto H^0(F(A)).
			\end{split}
		\end{equation}
		This formula simplifies to $\tau_{\geq 0} F(A)$ ($\tau_{\leq 0} F(A)$) is $F$ if right (left) t-exact -- in which case $F^{\heartsuit}$ is a right (left) exact functor between abelian categories.
	\end{proposition}
	\begin{proof}
	    Let
	    \[
	    0 \to A \to B \to C \to 0
	    \]
	    be a short exact sequence in $\cat T^\heartsuit$. This induces a distinguished triangle in $\cat T$, and hence a distinguished triangle
	    \[
	    F(A) \to F(B) \to F(C).
	    \]
	    We take the cohomological long exact sequence in $\cat S^\heartsuit$:
	    \[
	    H^{-1}(F(C)) \to H^0(F(A)) \to H^0(F(B)) \to H^0(F(C)) \to H^1(F(A)).
	    \]
	    If $F$ is right t-exact, then $H^1(F(A))\cong 0$; if it is left t-exact, then $H^{-1}(F(C))\cong 0$. From these observations, we easily conclude.
	\end{proof}
	\begin{proposition}\label{prop:left-right-texact}
		Let $\cat T$ and $\cat S$ be triangulated categories endowed with t-structures $(\cat T_{\leq 0}, \cat T_{\geq 0})$ and $(\cat S_{\leq 0}, \cat S_{\geq 0})$, and let $F \dashv G \colon \cat T \rightleftarrows \cat S$ be an adjunctions of triangulated functors. Then, $F$ is right t-exact if and only if $G$ is left t-exact.
	\end{proposition}
	\begin{proof}
	    Assume that $G$ is left t-exact, namely $G(\cat S_{\geq 0}) \subseteq \cat T_{\geq 0}$. Playing with shifts, we notice that this implies that $G(\cat S_{\geq n}) \subseteq \cat T_{\geq n}$ for all $n \in \mathbb Z$. Now, let $A \in \cat T_{\leq 0}$. Recall that $F(A) \in \cat S_{\leq 0}$ is equivalent to
    \begin{equation*}
    \cat S(F(A),B) = 0, \quad \forall\, B \in \cat S_{>0}.
    \end{equation*}
    On the other hand, for any $B \in \cat S_{>0}$, we have
    \begin{equation*}
    \cat S(F(A),B) \cong \cat T(A,G(B)) \cong 0,
    \end{equation*}
    since by hypothesis $G(B) \in \cat T_{>0}$. The other implication is proved in the same fashion.
	\end{proof}
	\begin{proposition} \label{prop:induced_adjunction_hearts}
		Let $F \dashv G \colon \cat T \rightleftarrows \cat S$ be an adjunctions of triangulated functors, and assume that $F$ is right t-exact (equivalently, $G$ is left t-exact). Then, there is an induced adjunction
		\begin{equation}
			F^\heartsuit \dashv G^\heartsuit \colon \cat T^\heartsuit \rightleftarrows \cat S^\heartsuit.
		\end{equation}
	\end{proposition}
	\begin{proof}
	    We may compute, for any $A \in \cat T^{\heartsuit}$ and $B \in \cat S^{\heartsuit}$:
        \begin{align*}
        \cat S^\heartsuit(F^\heartsuit(A),B)
        &\cong \cat S^\heartsuit(\tau_{\geq 0}F(A),B) && \text{($F$ is right t-exact)}\\
        &\cong \cat S(F(A),B) &&(B \in \cat S^\heartsuit \subseteq \cat S_{\geq 0}) \\
        &\cong \cat T(A,G(B)) \\
        &\cong \cat T(A,\tau_{\le 0}G(B)) &&(A \in \cat T^\heartsuit \subseteq \cat T_{\leq 0}) \\ 
        &\cong \cat T^\heartsuit(A,G^\heartsuit(B)) && \text{($G$ is left t-exact).}
         \qedhere
        \end{align*}
	\end{proof}

	\subsubsection{Epi-mono decompositions in t-structures}\label{subsubsec:epimono}
	In abelian categories, we can always factor a morphism as an epimorphism followed by a monomorphism. If we have a triangulated category with a t-structure, we can generalize such factorizations even quite outside the (abelian) heart.
	\begin{definition} \label{def:tmono_tepi}
		Let $\cat T$ be a triangulated category with a t-structure, and let $f \colon A \to B$ be a morphism in $\cat T$, with $A \in \cat T_{\leq 0}$. We say that $f$ is a \emph{t-monomorphism} if $\cone(f) \in \cat T_{\geq 0}$. We will sometimes picture a t-monomorphism as a hooked arrow $f \colon A \hookrightarrow B$.
		
		Moreover, let $g \colon A' \to B'$ be a morphism in $\cat T$, with both $A', B' \in \cat T_{\leq 0}$. We say that $g$ is a \emph{t-epimorphism} if $\cone(g)[-1] \in \cat T_{\leq 0}$. We will sometimes picture a t-epimorphism as an arrow with two heads $g \colon A' \twoheadrightarrow B'$.
	\end{definition}
	\begin{remark} \label{remark:tmono_tepi}
	    Using the long exact sequence in cohomology, it is easy to see that if $f \colon A \to B$ is a t-monomorphism, then $H^0(f)$ is a monomorphism and $H^{-n}(f)$ is an isomorphism for all $n \geq 1$. Analogously, we can check that if $g \colon A' \to B'$ is a t-epimorphism, then $H^0(f)$ is an epimorphism in $\cat T^\heartsuit$.
	\end{remark}
	 It may seem that the above Definition \ref{def:tmono_tepi} is quite ``ad hoc'' and asymmetrical; indeed, it looks like that t-monomorphisms and t-epimorphisms are not well-behaved concepts in general, but they need some additional restrictions to work.
	
	In the next result, we prove a ``t-epi/t-mono'' factorization for suitable morphisms in a triangulated category with a t-structure. A dual version appears in the proof of \cite[Proposition C.5.2.8]{lurie-SAG}; we give a detailed proof here, since it is one of the key ingredients of our main result.
	\begin{proposition} \label{prop:epimono_leftaisle}
		Let $\cat T$ be a triangulated category with a t-structure, and let $f \colon A \to B$ be a morphism in $\cat T$, with $A \in \cat T_{\leq 0}$. Then, $f$ can be decomposed as
		\begin{equation} \label{eq:epimono_leftaisle}
			\begin{tikzcd}
A \arrow[r, "p", two heads] & \operatorname{Coim}(f) \arrow[r, "\tilde{f}", hook] & B,
\end{tikzcd}
		\end{equation}
		where $\operatorname{Coim}(f) \in \cat T_{\leq 0}$, $p$ is a t-epimorphism (namely, $\cone(p)[-1] \in \cat T_{\leq 0}$) and $\tilde{f}$ is a t-monomorphism (namely, $\cone(\tilde{f}) \in \cat T_{\geq 0}$). In particular, taking $H^0$ (and recalling Remark \ref{remark:tmono_tepi}), this decomposition induces the usual epi-mono decomposition of $H^0(f)$ in the heart.
		
		More specifically, if we define $F$ as the object sitting in the distinguished triangle
		\[
		F \to A \xrightarrow{f} B,
		\]
		there is a distinguished triangle
		\begin{equation} \label{eq:epimono_leftaisle_coim}
		\tau_{\leq 0} F \to A \xrightarrow{p} \operatorname{Coim}(f)
		\end{equation}
		and a morphism of distinguished triangles
		\[
		\begin{tikzcd}
			\tau_{\leq 0} F \arrow[r] \arrow[d] & A \arrow[r, "p"] \arrow[d, equal] & \operatorname{Coim}(f) \arrow[r] \arrow[d, "\tilde{f}"] & {(\tau_{\leq 0} F) [1]} \arrow[d] \\
			F \arrow[r]                         & A \arrow[r, "f"]                                & B \arrow[r]                                        & {F[1],}                           
		\end{tikzcd}
		\]
		where $\tau_{\leq 0} F \to F$ is the natural morphism.
	\end{proposition}
	\begin{proof}
		In abelian categories, the coimage of $f$ is the cokernel of its kernel. We need to reinterpret this in the t-structure setting. Consider the distinguished triangle
		\[
		F \to A \xrightarrow{f}  B.
		\]
		The composition
		\[
		j \colon \tau_{\leq 0} F \to F \to A
		\]
		plays the role of the kernel of $f$. Hence, we may form the distinguished triangle:
		\[
		\tau_{\leq 0} F \xrightarrow{j} A \xrightarrow{p} \operatorname{Coim}(f).
		\]
		Since both $\tau_{\leq 0} F$ and $A$ lie in $\cat T_{\leq 0}$, the same is true for $\operatorname{Coim}(f)$, and moreover we have that $p$ is a t-epimorphism by construction. Consider the following commutative diagram with rows being distinguished triangles:
		\[
		\begin{tikzcd}
			\tau_{\leq 0} F \arrow[r] \arrow[d] & A \arrow[r, "p"] \arrow[d, equal] & \operatorname{Coim}(f) \arrow[r] \arrow[d, dotted, "\tilde{f}"] & {(\tau_{\leq 0} F) [1]} \arrow[d] \\
			F \arrow[r]                         & A \arrow[r, "f"]                                & B \arrow[r]                                        & {F[1]}                           
		\end{tikzcd}
		\]
		The morphism $\tilde{f}$ is obtained by the completing the leftmost commutative square to a morphism of distinguished triangles. Finally, recalling the natural distinguished triangle
		\[
		\tau_{\leq 0} F \to F \to \tau_{\geq 1} F,
		\]
		an application of the octahedral axiom (see \cite[Proposition 1.4.6]{neeman-triangulated}) yields a distinguished triangle:
		\[
		\tau_{\leq 0} F \to F \to \cone(\tilde{f})[-1],
		\]
		whence we deduce that $\tau_{\geq 1} F \cong \cone(\tilde{f})[-1]$ and that $\cone(\tilde{f})$ lies indeed in $\cat T_{\geq 0}$, as claimed.
	\end{proof}
	\begin{corollary} \label{corollary:tepi_coimage}
	    In the setting of the above Proposition \ref{prop:epimono_leftaisle}, assume moreover that $B \in \cat T_{\leq 0}$ and that $f \colon A \to B$ is a t-epimorphism. Then, any morphism $\tilde{f}$ in the decomposition \eqref{eq:epimono_leftaisle} is an isomorphism, and in particular $\operatorname{Coim}(f) \cong B$.
	\end{corollary}
    \begin{proof}
    By definition, in the proof of the above Proposition \ref{prop:epimono_leftaisle}, we find that the object $F$ lies in the left aisle $\cat T_{\leq 0}$. Hence, $\tau_{\leq 0} F \to F$ is an isomorphism and the result easily follows.
    \end{proof}
    \subsection{Derived injectives} \label{sect:dginj}
    T-structures allow to generalize the notion of injective object to the derived framework. \emph{Derived injectives} can be used to make resolutions of objects, just as their ``ordinary'' counterparts: this is thoroughly explored in \cite{genovese-lowen-vdb-dginj}. In this paper, we use them to prove a straightforward but crucial ``t-structure version'' of the following well-known fact:
    \begin{proposition} \label{prop:rightadj_preserve_inj_leftadj_exact_abelian}
        Let $G \colon \mathfrak A \to \mathfrak B$ be a functor between abelian categories which admits a left adjoint $F \colon \mathfrak B \to \mathfrak A$. Assume that $\mathfrak A$ has enough injectives. Then, $F$ is exact if and only if $G$ preserves injective objects.
    \end{proposition}
	\subsubsection{Basic definitions and properties}
	We shall discuss derived injectives and some of their features essentially following \cite[\S 5.1]{rizzardo-vdb-nonFM}. It is worth mentioning that analogous notions have appeared in literature, for instance \emph{injective objects of stable $\infty$-categories} (see \cite[C.5.7]{lurie-SAG}) or \emph{Ext-injective objects} (see \cite{assem-extproj}); derived injectives over non-positively graded dg-algebras, with respect to the canonical t-structure, have been investigated in \cite{shaul-injectivedgmod}. We refer to Appendix \ref{appendix:dginj} for a ``Baer-like'' characterization of derived injectives.
	\begin{definition} \label{def:dginj}
        Let $\cat T$ be an ($H^0(\basering k$)-linear) triangulated category with a t-structure $(\cat T_{\leq 0}, \cat T_{\geq 0})$, and let $I \in \operatorname{Inj}(\cat T^\heartsuit)$ be an injective object in the heart. The \emph{derived injective} associated to $I$ is an object $L(I)$ (uniquely determined up to isomorphism) which represents the cohomological functor $\cat T^\heartsuit(H^0(-),I) \colon \opp{\cat T} \to \Mod(H^0(\basering k))$:
        \begin{equation} \label{eq:dginj}
        \cat T^\heartsuit(H^0(-),I) \cong \cat T(-,L(I)).
        \end{equation}
        
        An object $E \in \cat T$ will be called \emph{derived injective} if there exists an injective object $I \in \cat T^\heartsuit$ such that $E \cong L(I)$.
        
        If for any $I \in \operatorname{Inj}(\cat T^\heartsuit)$ an object $L(I)$ as above exists, we say that \emph{$\cat T$ has derived injectives}. Moreover, if $\cat T$ has derived injectives and the heart $\cat T^\heartsuit$ has enough injectives, we say that \emph{$\cat T$ has enough derived injectives}. 
        
        If $\cat A$ is a ($\basering k$-linear) t-dg-category (cf. Definition \ref{def:tstruct_dgcat}), we shall say that \emph{$\cat A$ has (enough) derived injectives} if $H^0(\cat A)$ has this property.
	\end{definition}
	\begin{proposition} \label{prop:dginj_basicproperties}
	     Let $\cat T$ be an ($H^0(\basering k$)-linear) triangulated category with a t-structure $(\cat T_{\leq 0}, \cat T_{\geq 0})$, and let $I \in \operatorname{Inj}(\cat T^\heartsuit)$ be an injective object in the heart. Assume that the derived injective $L(I)$ associated to $I$ exists. Then:
	     \begin{enumerate}
	         \item \label{item:dginj_basicproperties_1} $L(I) \in \cat T_{\geq 0}$. 
	         \item \label{item:dginj_basicproperties_2} $H^0(L(I)) \cong I$.
	         \item \label{item:dginj_basicproperties_3} The functor $H^0 \colon \cat T \to \cat T^\heartsuit$ induces an isomorphism
	         \[
	         H^0 \colon \cat T(A,L(I)) \xrightarrow{\sim} \cat T^\heartsuit(H^0(A),I),
	         \]
	         for all $A \in \cat T$.
	     \end{enumerate}
	\end{proposition}
	\begin{proof}
	    To prove \eqref{item:dginj_basicproperties_1}, let $Z \in \cat T_{\leq -1}$. Then, $H^0(Z) \cong 0$, so:
	    \[
	    \cat T(Z,L(I)) \cong \cat T^\heartsuit(H^0(Z),I) \cong 0,
	    \]
	    and we conclude that indeed $L(I) \in \cat T_{\geq 0}$.
	    
	    For \eqref{item:dginj_basicproperties_2}, let $B \in \cat T^\heartsuit$. We have natural isomorphisms:
	    \begin{align*}
	        \cat T^\heartsuit(B, H^0(L(I))) & \cong \cat T(B,\tau_{\leq 0} L(I)) && \text{(use point \eqref{item:dginj_basicproperties_1})} \\
	        & \cong \cat T(B, L(I)) \\
	        & \cong \cat T^\heartsuit(B, I) && \text{(use that $B \in \cat T^\heartsuit$ and the def. of $L(I)$).}
	    \end{align*}
	    By the Yoneda lemma in $\cat T^\heartsuit$, we obtain an isomorphism $H^0(L(I)) \cong I$.
	    
	    For \eqref{item:dginj_basicproperties_3}, we consider the following diagram:
	    \[
	    \begin{tikzcd}
        {\cat T(A,L(I))} \arrow[r, "H^0"] \arrow[rd, "\eqref{eq:dginj}"'] & {\cat T^\heartsuit(H^0(A),H^0(L(I)))} \arrow[d, "\sim"] \\
                                                     & {\cat T^\heartsuit(H^0(A),I).}                  
        \end{tikzcd}
	    \]
	    The diagonal arrow of the diagram is the natural isomorphism \eqref{eq:dginj} which defines $L(I)$; the vertical arrow is the natural isomorphism we described in the proof of point \eqref{item:dginj_basicproperties_2}. A direct inspection shows that the above diagram is commutative, and we conclude. 
	\end{proof}
	\begin{remark} \label{remark:dginj_H0inj_uniquely}
	    Let $E \in \cat T$ be a derived injective, namely, there exists an injective object $I \in \cat T^\heartsuit$ such that $E \cong L(I)$ (cf. Definition \ref{def:dginj}). Then, by point \eqref{item:dginj_basicproperties_2} of the above Proposition \ref{prop:dginj_basicproperties}, we see that $I \cong H^0(E)$. We conclude that \emph{$I \cong H^0(E)$ is an injective object uniquely determined by $E$}.
	\end{remark}
	
	We can characterize derived injectives by a vanishing property -- this is the actual link with Ext-injective objects that we mentioned before.
	\begin{proposition} \label{prop:dginj_equivalent_definitions}
	    Let $\cat T$ be a triangulated category with a t-structure, and let $E \in \cat T$ be an object. The following are equivalent:
	    \begin{enumerate}
	        \item \label{item:dginj_equivalent_definitions_1} $E$ is a derived injective (Definition \ref{def:dginj}).
	        \item \label{item:dginj_equivalent_definitions_2} $E \in \cat T_{\geq 0}$ and for any $Z \in \cat T_{\geq 0}$ we have
	        \begin{equation}
	            \cat T(Z[-1],E) \cong 0. \label{eq:dginj_vanishing}
	        \end{equation}
	    \end{enumerate}
	\end{proposition}
	\begin{proof}
	    The implication $\eqref{item:dginj_equivalent_definitions_1} \Rightarrow \eqref{item:dginj_equivalent_definitions_2}$ follows from the definition and from Proposition \ref{prop:dginj_basicproperties} \eqref{item:dginj_basicproperties_1}.
	    
	    Let us now prove $\eqref{item:dginj_equivalent_definitions_2} \Rightarrow \eqref{item:dginj_equivalent_definitions_1}$. First, we check that $H^0(E)$ is an injective object in $\cat T^\heartsuit$. Indeed, let
	    \[
	    0 \to A \to B \to C \to 0
	    \]
	    be a short exact sequence in $\cat T^\heartsuit$. This induces a distinguished triangle
	    \begin{equation} \label{eq:dginj_equivalent_definitions_ses}
	    A \to B \to C \tag{$\ast$}
	    \end{equation}
	    in $\cat T$. Moreover, since $E \in \cat T_{\geq 0}$, we have:
	    \begin{equation} \label{eq:dginj_equivalent_definitions_natiso} \tag{$\bullet$}
	    \begin{split} 
	        \cat T^\heartsuit(-,H^0(E)) & \cong \cat T^\heartsuit(-,\tau_{\leq 0} E) \\
	        & \cong \cat T(-,E)
	    \end{split}
	    \end{equation}
	    as functors $\opp{(\cat T^\heartsuit)} \to \Mod(H^0(\basering k))$. We consider the following long exact sequence induced by \eqref{eq:dginj_equivalent_definitions_ses}:
	    \[
	    \cat T(A[1],E) \to \cat T(C,E) \to \cat T(B,E) \to \cat T(A,E) \to \cat T(C[-1], E).
	    \]
	    Since $\cat C \in \cat T^\heartsuit \subseteq \cat T_{\geq 0}$, we have by hypothesis that $\cat T(C[-1],E)=0$; furthermore, $\cat T(A[1], E)=0$ since $A[1] \in \cat T_{\leq -1}$ and $E \in \cat T_{\geq 0}$. Recalling the above \eqref{eq:dginj_equivalent_definitions_natiso}, we see that the sequence
	    \[
	    0 \to \cat T^\heartsuit(C,H^0(E)) \to \cat T^\heartsuit(B,H^0(E)) \to \cat T^\heartsuit(A, H^0(E)) \to 0
	    \]
	    is exact, as we wanted.
	    
	    Next, we show that $E$ is the derived injective associated to the injective object $H^0(E)$, by checking that the morphism
	    \[
	    H^0 \colon \cat T(X, E) \to \cat T^\heartsuit(H^0(X),H^0(E))
	    \]
	    is an isomorphism for all $X \in \cat T$. Since $E \in \cat T_{\geq 0}$, we have an isomorphism
	    \[
	    \cat T(X,E) \cong \cat T(\tau_{\geq 0} X, E).
	    \]
	    Then, consider the canonical distinguished triangle
	    \[
	    H^0(X) \to \tau_{\geq 0} X \to \tau_{\geq 1} X
	    \]
	    and the induced exact sequence:
	    \[
	     \cat T(\tau_{\geq 1} X,E) \to \cat T(\tau_{\geq 0} X, E) \to \cat T(H^0(X), E) \to \cat T((\tau_{\geq 1}X)[-1], E).
	    \]
	    The object $\tau_{\geq 1}X$ lies in $\cat T_{\geq 1} \subseteq \cat T_{\geq 0}$, and moreover we can write $\tau_{\geq 1}X \cong (\tau_{\geq 1} X)[1][-1]$, and $(\tau_{\geq 1} X)[1] \in \cat T_{\geq 0}$. Hence, by hypothesis, we deduce that
	    \begin{align*}
	    \cat T(\tau_{\geq 1} X, E) &\cong 0, \\
	    \cat T((\tau_{\geq 1}X)[-1], E) & \cong 0.
	    \end{align*}
	    Hence, we have an isomorphism
	    \[
	    \cat T(\tau_{\geq 0} X, E) \cong \cat T(H^0(X),E).
	    \]
	    Finally, we have an isomorphism
	    \[
	    \cat T(H^0(X),E) \cong \cat T(H^0(X),H^0(E)).
	    \]
	    With a direct inspection, we now see that $H^0 \colon \cat T(X, E) \to \cat T^\heartsuit(H^0(X),H^0(E))$ can be factored as the chain of isomorphisms we have just described, and we conclude.
	\end{proof}
	We have an easy corollary which perhaps makes derived injectives a little more familiar.
	\begin{corollary} \label{coroll:dginj_extension_tmono}
	    Let $E$ be a derived injective, and let $g \colon A \to B$ be a t-monomorphism (Definition \ref{def:tmono_tepi}). Then, the induced morphism
	    \[
	    g^* \colon \cat T(B,E) \to \cat T(A,E)
	    \]
	    is surjective. In other words, any morphism $f \colon A \to E$ can be extended to a morphism $\tilde{f} \colon B \to E$ along the t-monomorphism $g \colon A \to B$:
	    \begin{equation} \label{eq:dginj_extension_tmono}
            \begin{tikzcd}
            A \arrow[r, "g", hook] \arrow[d, "f"'] & B \arrow[ld, "\tilde{f}", dotted] \\
            E.                                      &                                  
            \end{tikzcd}
	    \end{equation}
	\end{corollary}
	\begin{proof}
	   We consider the distinguished triangle
	    \[
	    A \xrightarrow{g} B \to \cone(g).
	    \]
	    By assumption, $\cone(g) \in \cat T_{\geq 0}$, and by \Cref{prop:dginj_equivalent_definitions} above we know that 
	    \[
	    \cat T(\cone(g)[-1], E) \cong 0.
	    \]
	    We now immediately conclude by considering the exact sequence
	    \[
	    \cat T(B, E) \xrightarrow{g^*} \cat T(A,E) \to \cat T(\cone(g)[-1], E). \qedhere
	    \]
	\end{proof}
	We now give a reasonable sufficient condition which ensures that a given triangulated category with a t-structure \emph{has derived injectives}, in the sense of Definition \ref{def:dginj}.
	\begin{proposition} \label{prop:dginj_brownrepresentability}
		Let $\cat T$ be a triangulated category with a t-structure. In addition, assume that $\cat T$ is well generated \cite[Def 8.1.4]{neeman-triangulated} (in particular, it is closed under small direct sums). Assume moreover that the cohomological functor 
		\[
		H^0 \colon \cat T \to \cat T^\heartsuit
		\]
		preserves small direct sums. Then, $\cat T$ has derived injectives (cf. Definition \ref{def:dginj}).
	\end{proposition}
	\begin{proof}
		Let $I$ be an injective object in $\cat T^\heartsuit$. By hypothesis, the functor 
		\[
		\cat T^\heartsuit(H^0(-),I)
		\]
		is cohomological and maps small direct sums to direct products. As Brown representability holds for well generated triangulated categories (see, for example, \cite[Thm A]{krause-brownrepresentability}), we conclude that $\cat T^\heartsuit(H^0(-),I)$ is representable, as desired.
	\end{proof}
	\subsubsection{Preservation of derived injectives and t-exactness}\label{subsubsec:texactness-dginj}
    In this part we prove the analogue of Proposition \ref{prop:rightadj_preserve_inj_leftadj_exact_abelian} in the framework of t-structures, as promised. However, before stating it, we need an additional hypothesis which essentially ensures that objects are uniquely determined by their cohomologies.
    \begin{definition} \label{def:tstruct_nondegenerate}
        Let $\cat T$ be a triangulated category with a t-structure. We say that such t-structure is \emph{non-degenerate} if one of the following equivalent conditions holds: 
        \begin{enumerate}
            \item \label{def:tstruct_nondegenerate_1} For any object $A \in \cat T$, we have that $A \cong 0$ if and only if $H^n_t(A) \cong 0$ for all $n \in \mathbb Z$.
            \item \label{def:tstruct_nondegenerate_2} Any morphism $f \colon A \to B$ in $\cat T$ is an isomorphism if and only if $H^n_t(f) \colon H^n_t(A) \to H^n_t(B)$ is an isomorphism in $\cat T^\heartsuit$, for all $n \in \mathbb Z$.
        \end{enumerate}
        
        If $\cat A$ is a t-dg-category, we say that it (or its t-structure) is \emph{non-degenerate} if the t-structure on the homotopy category $H^0(\cat A)$ is non-degenerate.
    \end{definition}
    \begin{remark} \label{remark:tstruct_nondegerate_aisles}
        If $\cat T$ has a non-degenerated t-structure, it is straightforward to see that its aisles $\cat T_{\leq 0}$ and $\cat T_{\geq 0}$ are completely determined by cohomologies, namely:
        \begin{equation} \label{eq:tstruct_nondegerate_aisles}
            \begin{split}
                    \cat T_{\leq n} &= \{ X \in \cat T : H^k_t(X) = 0,\ \forall\, k > n \}, \\
                    \cat T_{\geq n} &= \{ X \in \cat T : H^k_t(X) = 0,\ \forall\, k < n \},
            \end{split}
        \end{equation}
        as full subcategories of $\cat T$.
        
        In this setting, we then easily deduce that $f \colon A \to B$ (with $A \in \cat T_{\leq 0}$) is a t-monomorphism (Definition \ref{def:tmono_tepi}) if and only if $H^0_t(f)$ is a monomorphism and $H^{-n}_t(f)$ is an isomorphism for all $n \geq 1$; analogously, a morphism $g \colon A' \to B'$ (with $A', B' \in \cat T_{\leq 0}$) is a t-epimorphism if and only if $H^0_t(g)$ is an epimorphism.
    \end{remark}
	\begin{proposition}\label{prop:adjoint-dginj}
	    Let $\cat T$ and $\cat S$ be triangulated categories with t-structures, and assume that $\cat T$ has a non-degenerate t-structure (Definition \ref{def:tstruct_nondegenerate}) and enough derived injectives (Definition \ref{def:dginj}). Moreover, let $F\dashv G: \cat S \rightleftarrows \cat T$ be an adjunction of triangulated functors, such that $G$ is left t-exact. Then, $F$ is t-exact if and only if $G$ preserves derived injectives.
	\end{proposition}
		\begin{proof}
		    Since $G$ is left t-exact, we know from Proposition \ref{prop:left-right-texact} that $F$ is right t-exact. Hence, it is enough to prove that $F$ is left t-exact if and only if $G$ preserves derived injectives.
		    
			Assume that $F$ is left t-exact, and let $E \in \cat T$ be a derived injective. As $E \in \cat T_{\geq 0}$ and $G$ is left t-exact by hypothesis, we have that $G(E) \in \cat S_{\geq 0}$. By assumption, if $C \in \cat S_{\geq 0}$ then $F(C) \in \cat T_{\geq 0}$. Consequently, we have that  
			\begin{equation*}
				\cat S(C[-1],G(E)) \cong \cat T(F(C)[-1],E) = 0
			\end{equation*}
			for every $C \in \cat S_{\geq 0}$, which proves that $G(E)$ is derived injective (cf. Proposition \ref{prop:dginj_equivalent_definitions}).
			
			On the other hand, suppose that $G$ preserves derived injectives and let $C \in \cat S_{\geq 0}$. The t-structure on $\cat T$ is non-degenerate by assumption, so (recalling the above Remark \ref{remark:tstruct_nondegerate_aisles}) it is enough to show that
			\[
			H^0(F(C)[-n]) = H^{-n}(F(C)) \cong 0
			\]
			in $\cat T^{\heartsuit}$, for all $n \geq 1$. First, we observe that for every derived injective $E \in \cat T$ we have that 
			\begin{equation} \label{eq:homszero} \tag{$\ast$}
				\cat T(F(C)[-n], E) \cong \cat S(C[-n],G(E)) = 0,
			\end{equation} 
			for every $n \geq 1$, since $C[-n] \in \cat S_{\geq 1}$ (cf. Proposition \ref{prop:dginj_equivalent_definitions}). As $\cat T$ has enough derived injectives, the heart $\cat T^{\heartsuit}$ has enough injectives. Consequently, there exists an injective object $I$ in $\cat T^{\heartsuit}$ and a monomorphism $i \in \cat T^{\heartsuit}(H^{-n}(F(C)), I)$, for any fixed $n \geq 1$. In addition, again because $\cat T$ has enough derived injectives, we have that $I$ has an associated derived injective $L(I)$; in particular, there is an isomorphism:
			\[
			\cat T^{\heartsuit} (H^{-n}(F(C)), I) \cong \cat T(F(C)[-n], L(I)).
			\]
			This, taking $E=L(I)$ in \eqref{eq:homszero} above, implies that
			\begin{equation*}
				i \in \cat T^{\heartsuit}(H^{-n}(F(C)), I) \cong \cat T(F(C)[-n],L(I)) = 0.
			\end{equation*}
			As $i$ is a monomorphism, we deduce that $H^{-n}(F(C)) = 0$ ($n\geq 1$), as we wanted.
		\end{proof}
	
	\section{A derived Gabriel-Popescu theorem for t-structures}\label{sect:Gabriel-Popescu}
	This part is devoted to the proof of the main result of the paper, namely, a derived Gabriel-Popescu theorem for dg-categories endowed with a t-structure (also called t-dg-categories as in Definition \ref{def:tstruct_dgcat}).
	\subsection{The setup}
	    The classical Gabriel-Popescu theorem shows that Grothendieck abelian categories are localizations of categories of modules \cite{gabriel-popescu-original}. In \cite{porta-gabrielpopescu-triangulated} a triangulated Gabriel-Popescu theorem is provided, showing that well generated algebraic triangulated categories are localizations of derived categories of small dg-categories. Our aim is to provide a t-dg-Gabriel-Popescu theorem for t-dg-categories with suitable assumptions which capture the idea of ``being Grothendieck in a derived sense''. In particular, the heart of such t-dg-categories will be a Grothendieck abelian category.
	\begin{setup} \label{setup:GabrielPopescu}
	We shall fix a $\basering k$-linear t-dg-category $\cat A$. We make the following assumptions:
	\begin{itemize}
		\item $H^0(\cat A)$ is well generated. In particular, it has small direct sums.
		\item The t-structure on $\cat A$ is non degenerate (Definition \ref{def:tstruct_nondegenerate}).
		\item The cohomological functor
		\[
		H^0_t = H^0 \colon H^0(\cat A) \to H^0(\cat A)^\heartsuit
		\]
		preserves small direct sums.
		\item Filtered colimits are exact in $H^0(\cat A)^\heartsuit$.
		\item The t-dg-category $\cat A$ has a small \emph{set of generators}, namely, a set $\mathcal U$ of objects in $\cat A_{\leq 0}$, such that for every object $A \in \cat A_{\leq 0}$ there is a t-epimorphism in $H^0(\cat A)$ (Definition \ref{def:tmono_tepi}):
		\[
		p \colon \bigoplus_{i \in I} U_i \twoheadrightarrow A,
		\]
		where $U_i \in \mathcal U$ for all $i \in I$, for some set $I$.
	\end{itemize}
	\end{setup}
	\begin{remark}
	The assumptions of the above Setup \ref{setup:GabrielPopescu} are slightly different than the ones of \cite[C.2.1.10]{lurie-SAG}, but essentially equivalent. The advantage of our set of assumptions is that they only depend on the homotopy category $H^0(\cat A)$.
	\end{remark}
	\begin{lemma} \label{lemma:heart_Grothendieck}
	Let $\cat A$ be as in Setup \ref{setup:GabrielPopescu} above. Then, the heart $H^0(\cat A)^\heartsuit$ is a Grothendieck abelian category, with generators given by the set
	\[
	H^0_t(\mathcal U) = \{ H^0_t(U) : U \in \mathcal U\}.
	\]
	
	In particular, $\cat A$ has enough derived injectives (Definition \ref{def:dginj}).
	\end{lemma}
	\begin{proof}
	First, the assumption that $H^0_t$ preserves direct sums entails that, for any set of objects $\{B_j : j \in J\}$ in $H^0(\cat A)^\heartsuit$, the object $H^0_t (\oplus_j B_j)$ is indeed (together with the suitable structure morphisms) a direct sum of the $B_j$ in $H^0(\cat A)^\heartsuit$. Hence, the heart $H^0(\cat A)^\heartsuit$ has direct sums (and in particular all colimits).
	
	Next, let $A \in H^0(\cat A)^\heartsuit$. By hypothesis, there is a t-epimorphism
	\[
	p \colon \bigoplus_{i \in I} U_i \twoheadrightarrow A,
	\]
	where $U_i \in \mathcal U$ for all $i \in I$, for some set $I$. By definition, this induces an epimorphism
	\[
	H^0_t(p) \colon \bigoplus_{i \in I} H^0(U_i) \twoheadrightarrow H^0_t(A)=A,
	\]
	where we also used that $H^0_t$ preserves direct sums. This shows that $H^0_t(\mathcal U)$ is a set of generators of $H^0(\cat A)^\heartsuit$. Exactness of filtered colimits holds by assumption, so we conclude that $H^0(\cat A)^\heartsuit$ is indeed a Grothendieck abelian category.
	
	By Proposition \ref{prop:dginj_brownrepresentability} we know that $\cat A$ has derived injectives. Since $H^0(\cat A)^\heartsuit$ has enough injectives (being Grothendieck abelian), we conclude that $\cat A$ has indeed enough derived injectives.
	\end{proof}
	\begin{example} \label{example:derivedcategory}
	     Let $\smallcat a$ be a (small) dg-category concentrated in nonpositive degrees:
	     \[
	     H^i(\smallcat a(A,B)) \cong 0, \qquad A,B \in \smallcat a, \quad i > 0.
	     \]
	     The derived dg-category $\dercompdg(\smallcat a)$ is naturally a t-dg-category (cf. Proposition \ref{prop:dercomp_naturaltstruct}). Moreover, it is easy to see that it satisfies the assumptions of the above Setup \ref{setup:GabrielPopescu}:
	     \begin{itemize}
	         \item The derived category $\dercomp(\smallcat a)=H^0(\dercompdg(\smallcat a))$ is compactly generated and the representables $\smallcat a(-,A)$ constitute a set of compact generators. Actually, all algebraic triangulated categories that are compactly generated are of this form (see, for example, \cite[Thm Cor 7.3]{porta-gabrielpopescu-triangulated}). In particular, $\dercomp(\smallcat a)$ is well generated.
	         \item The t-structure on $\dercompdg(\smallcat a)$ is clearly non-degenerate: acyclic objects are isomorphic to $0$ in the derived category, practically by definition.
	         \item Using h-projective resolutions it is easy to see that the cohomological functor
	         \[
	         H^0 \colon \dercomp(\smallcat a) \to \Mod(H^0(\smallcat a))
	         \]
	         preserves (small) direct sums. 
	         \item The heart of the t-structure is $\Mod(H^0(\smallcat a))$, which is a Grothendieck category. In particular, filtered colimits are exact in the heart .
	         \item $\dercompdg(\smallcat a)$ has enough derived injectives. This follows from Proposition \ref{prop:dginj_brownrepresentability} and the fact that the heart $\Mod(H^0(\smallcat a))$ has enough injectives.
	         \item The representables $\smallcat a(-,A)$ form a set of generators of the canonical t-structure. Indeed, for any object $M \in \dercomp(\smallcat a)_{\leq 0}$, we can find an epimorphism
	         \[
	         \bigoplus_{i \in I} H^0(\smallcat a(-,A_i)) \twoheadrightarrow H^0(M),
	         \]
	         for a suitable family of objects $\{A_i : i \in I\}$. By the Yoneda lemma, we can lift this epimorphism to a t-epimorphism
	         \[
	         \bigoplus_{i \in I} \smallcat a(-,A_i) \twoheadrightarrow M.
	         \]
	     \end{itemize}
	     The derived dg-category $\dercompdg(\smallcat a)$ is a prototypical example of t-dg-category for which the derived Gabriel-Popescu theorem holds, and it is the derived counterpart of the Grothendieck abelian category of modules.
	\end{example}
	\begin{example} \label{example:dercat_Grothendieckabelian}
	In thi s example we assume that $\basering k$ is an ordinary commutative ring. 
	
	Let $\mathfrak G$ be a Grothendieck $\basering k$-linear category. It is well-known that the derived category $\dercomp (\mathfrak G)$ admits a unique dg-enhancement \cite[Thm A]{canonaco-stellari-dgenh-survey} (see also \S\ref{subsec:application}) that we denote by $\dercompdg (\mathfrak G)$. The dg-category $\dercompdg (\mathfrak G)$ is naturally a t-dg-category by considering the standard t-structure of $\dercomp(\mathfrak G)$.
	In addition, it satisfies the assumptions provided in \Cref{setup:GabrielPopescu}:
	\begin{itemize}
	    \item The derived category $\dercomp(\mathfrak G) = H^0(\dercompdg(\mathfrak G))$ is well generated \cite[Thm 5.10]{krause-derivingauslander}. More concretely, we can always find a regular cardinal $\alpha \neq \aleph_0$ such that $\mathfrak G$ is locally $\alpha$-presentable and its full subcategory $\mathfrak G^{\alpha}$ of $\alpha$-presentable objects is abelian. Then $\dercomp(\mathfrak G)$ is $\alpha$-compactly generated and the $\alpha$-compact generators are given by $\dercomp(\mathfrak G^{\alpha})$ (see \cite{krause-derivingauslander}).
	    \item As in \Cref{example:derivedcategory} above, the t-structure is trivially non-degenerate: acyclic objects are isomorphic to $0$ in the derived category. 
	    \item The heart of the standard t-structure on $\dercomp(\mathfrak G)$ is given by $\mathfrak G$. In particular, because $\mathfrak G$ is Grothendieck, we have that filtered colimits are exact in the heart.
	    \item The cohomological functor 
	    \[
	    H^0: \dercomp(\mathfrak G) \to \mathfrak G
	    \]
	    preserves (small) direct sums as a direct consequence of the fact that filtered colimits are exact in $\mathfrak G$.
	    \item $\dercompdg(\mathfrak G)$ has enough derived injectives. This follows from \Cref{prop:dginj_brownrepresentability} and the fact that $\mathfrak G$ has enough injectives.
	    \item Let $U$ be a generator of $\mathfrak G$ and consider it as a chain complex concentrated in degree 0. In particular, we can see $U$ as an object in $\dercompdg(\mathfrak G)_{\leq 0}$.  We show that $U$ is a generator of the standard t-structure.
	    Given a chain complex $B \in \dercomp(\mathfrak G)_{\leq 0}$, we may assume that $B$ is strictly concentrated in non-positive degrees by replacing it with its truncation $\tau_{\leq 0}B$. Then, because $U$ generates $\mathfrak G$, we can always find a morphism of chain complexes
	    \[
	    f: \bigoplus_{i \in I} U \to B
	    \]
	    that is surjective on degree 0. Observe that both $B$ and $\bigoplus_{i \in I} U$, seen as objects in $\dercomp(\mathfrak G)$, belong to $\dercomp(\mathfrak G)_{\leq 0}$. We are going to show that the morphism induced by $f$ in $\dercomp(\mathfrak G)$ is a t-epimorphism. Consider the distinguished triangle
	    \[
	    \bigoplus_{i \in I} U \to B \to \cone(f)
	    \]
	    that $f$ induces on $\dercomp(\mathfrak G)$.
	     By construction, we have that $\cone(f) \in \dercomp(\mathfrak G)_{\leq 0}$ and it only remains to show that $H^0(\cone(f)) = 0$ to conclude. We compute the long exact sequence of cohomology: 
	    \[
	    H^{-1}(\cone(f)) \to \bigoplus_{i \in I} U \twoheadrightarrow H^0(B) \to H^0(\cone(f)) \to 0 
	    \]
	    Observe that $\bigoplus_{i \in I} U \twoheadrightarrow H^0(B)$ is an epimorphism in $\mathfrak G$ by construction of $f$. Hence, we deduce that $H^0(\cone(f)) = 0$ as desired.
	\end{itemize}
	\end{example}
	\subsection{Statement of the main results (and a few corollaries)}
	We let $\cat A$ be a t-dg-category as in Setup \ref{setup:GabrielPopescu}. Moreover, we view the set of generators $\mathcal U$ as a full dg-subcategory of $\cat A$. Recalling \S \ref{subsubsec:truncations_dgcat}, we set:
	\[
	\smallcat u = \tau^{\leq 0} \mathcal U.
	\]
	The objects of $\smallcat u$ are the same as $\mathcal U$. There is a natural dg-functor
	\[
	j \colon \smallcat u \to \cat A.
	\]
	Recalling \S \ref{subsubsec:ind_res}, we have a restriction quasi-functor:
	\[
	\Res_j \colon \dercompdg(\cat A) \to \dercompdg(\smallcat u).
	\]
	Composing it with the Yoneda embedding $\cat A \hookrightarrow \dercompdg(\cat A)$, we obtain a quasi-functor
	\begin{equation} \label{eq:GabrielPopescu_rightadj_def}
	    G \colon \cat A \to \dercompdg(\smallcat u).
	\end{equation}
	At the level of homotopy categories, this can be identified with the triangulated functor
	\begin{equation} \label{eq:GabrielPopescu_rightadj_hocat}
	    \begin{split}
	        H^0(\cat A) & \to \dercomp(\smallcat u), \\
	        A & \mapsto \cat A(j(-),A).
	    \end{split}
	\end{equation}
	
	We can finally state our main result.
	\begin{theorem}[A derived Gabriel-Popescu theorem for t-dg-categories]\label{thm:GabrielPopescu}
		Let $\cat A$ be a t-dg-category as in Setup \ref{setup:GabrielPopescu}. Then, the quasi-functor
		\begin{equation} \label{eq:GabrielPopescu_rightadj_def2}
			\begin{split}
				G \colon \cat A & \to \dercompdg(\smallcat u), \\
				A & \mapsto \cat A(j(-), A)
			\end{split}
		\end{equation}
	is quasi-fully faithful and has a t-exact left adjoint quasi-functor
		\begin{equation} \label{eq:GabrielPopescu_ladj}
				F \colon \dercompdg(\smallcat u) \to \cat A
		\end{equation}
	such that $H^0(F)(\smallcat u(-,U)) \cong U$ in $H^0(\cat A)$, naturally in $U \in H^0(\smallcat u)$.

	In particular, $(\cat A, F)$ is a dg-quotient (cf. \S \ref{subsubsec:dgquotients_localizations}) of $\dercompdg(\smallcat u)$ by its full dg-subcategory $\ker(F)$ spanned by the objects $\{M \in \dercompdg(\smallcat u) : H^0(F)(M) \cong 0\}$.
	
	Moreover, the adjunction $F \dashv G$ induces an adjunction
	\begin{equation} \label{eq:GabrielPopescu_adj_hearts}
	    H^0(F)^\heartsuit \dashv H^0(G)^\heartsuit \colon \Mod(H^0(\smallcat u)) \rightleftarrows H^0(\cat A)^\heartsuit,
	\end{equation}
	such that $H^0(F)^\heartsuit$ is exact and
	\begin{equation} \label{eq:GabrielPopescu_rightadj_heart}
	\begin{split}
	    	    H^0(G)^\heartsuit \colon H^0(\cat A)^\heartsuit & \to \Mod(H^0(\smallcat u)), \\
	    	    A & \mapsto H^0(\cat A)(j(-),A),
	\end{split}
	\end{equation}
	is fully faithful.
	\end{theorem}
	
	There is an obvious corollary of Theorem \ref{thm:GabrielPopescu} in the case where the set of generators $\mathcal U$ has only one object $U$. In that case, the dg-category $\smallcat u$ can be identified with the dg-algebra $R = \tau_{\leq 0} \cat A(U,U)$ and the Gabriel-Popescu theorem can be restated as follows:
	\begin{corollary} \label{coroll:GabrielPopescu_singleobj}
	    Let $\cat A$ be a t-dg-category as in Setup \ref{setup:GabrielPopescu}, and assume that it has a single generator $U$. Set
	    \begin{equation}
	        R = \tau_{\leq 0} \cat A(U,U).
	    \end{equation}
	    Then, the quasi-functor
	    \begin{equation}
	    \begin{split}
	        G \colon \cat A & \to \dercompdg(R), \\
	        A & \mapsto \cat A(U,A)
	    \end{split}
	    \end{equation}
	    is quasi-fully faithful and has a t-exact left adjoint quasi-functor
	    \[
	    F \colon \dercompdg(R) \to \cat A
	    \]
	    such that $H^0(F)(R)\cong U$ in $H^0(\cat A)$.
	    
	    In particular, $(\cat A, F)$ is a dg-quotient (cf. \S \ref{subsubsec:dgquotients_localizations}) of $\dercompdg(R)$ by its full dg-subcategory $\ker(F)$.
	    
	    Moreover, the adjunction $F \dashv G$ induces an adjunction
	    \[
	       H^0(F)^\heartsuit \dashv H^0(G)^\heartsuit \colon \Mod(H^0(\smallcat u)) \rightleftarrows H^0(\cat A)^\heartsuit,
	    \]
	    such that $H^0(F)^\heartsuit$ is exact and $H^0(G)^\heartsuit$ is fully faithful.
	\end{corollary}
	
	More interestingly, we can deduce the ordinary Gabriel-Popescu theorem from the derived one. We state it in the classical ``single generator'' case, but a ``many generators'' version can be deduced in the same way.
	\begin{corollary}[Gabriel-Popescu theorem]\label{coroll:classicalGP}
	    Assume that $\basering k$ is an ordinary commutative ring. Let $\mathfrak A$ be a ($\basering k$-linear) Grothendieck abelian category with a generator $U$. Set
	    \[
	    R = \mathfrak A(U,U).
	    \]
	    Then, there is a fully faithful functor
	    \begin{align*}
	        \mathfrak A & \to \Mod(R), \\
	        A & \mapsto \mathfrak A(U,A)
	    \end{align*}
	    which admits an exact left adjoint.
	\end{corollary}
	\begin{proof}
	By Example \ref{example:dercat_Grothendieckabelian}, the dg-derived category $\dercompdg(\mathfrak A)$ satisfies the assumptions of Setup \ref{setup:GabrielPopescu}, and is generated by the object $U \in \mathfrak A$ as a t-dg-category. Since $U$ is concentrated in degree $0$ in $\dercompdg(\mathfrak A)$, it is easy to see that 
	\begin{align*}
	\tau_{\leq 0}\dercompdg(\mathfrak A)(U,U) & \cong \dercomp(\mathfrak A)(U,U) \\
	& \cong \mathfrak A(U,U) \\ 
	&= R.
	\end{align*}
	Then, Theorem \ref{thm:GabrielPopescu} yields an adjunction
	\[
	F \dashv G \colon \dercompdg(R) \rightleftarrows \dercompdg(\mathfrak A)
	\]
	such that $F$ is t-exact and $G$ is quasi-fully faithful. Passing to hearts, we get an adjunction
	\[
	H^0(F)^\heartsuit \dashv H^0(G)^\heartsuit \colon \Mod(R) \rightleftarrows \mathfrak A,
	\]
	and we know that $H^0(F)^\heartsuit$ is exact and $H^0(G)^\heartsuit$ is fully faithful.
	\end{proof}
	\subsection{The proof of the derived Gabriel-Popescu theorem}
	The rest of the paper is devoted to prove Theorem \ref{thm:GabrielPopescu}. We will organize the different parts of the proof in subsections. The technical core of the proof will rely on a derived version of an argument used by Mitchell in his proof of the classical Gabriel-Popescu theorem: we present this result in \S\ref{subsect:Mitchell}. 
	
	From now on, we fix a t-dg-category $\cat A$ satisfying the assumptions of Setup \ref{setup:GabrielPopescu}.

\subsubsection{The existence of the left adjoint}
	In this subsection we prove the existence of the left adjoint of the quasi-functor \eqref{eq:GabrielPopescu_rightadj_def2}.
	\begin{proposition}\label{prop:existence-left-adjoint}
		Given $\cat A$ a dg-category as above and denote $\cat T = H^0(\cat A)$. Then the quasi-functor
		\begin{equation}
			\begin{split}
				G \colon \cat A & \to \dercompdg(\smallcat u), \\
				A & \mapsto \cat A(j(-),A)
			\end{split}
		\end{equation}
	has a left adjoint
		\begin{equation}
			F \colon \dercompdg(\smallcat u) \to \cat A
		\end{equation}
	such that $H^0(F)(\smallcat u(-,U)) \cong U$ in $H^0(\cat A)$ naturally in $U \in H^0(\smallcat u)$.
	\begin{proof}
		By Lemma \ref{lemma:adjoints_quasifunctors_cohomology} we know that if $H^0(G)$ has a left adjoint $L$, then there exists a quasi-functor $F$ which is left adjoint to $G$ and such that $H^0(F) = L$. From this, together with the fact that, for each $U \in \mathcal U$, $G(U)= \cat A (j(-),U) = \smallcat u(-,U)$ is h-projective, we conclude that to prove the claim it is enough to show that $H^0(G)$ has a left adjoint $L$ such that $L(\smallcat u(-,U)) = U$ for every $U \in \mathcal U$ naturally in $U \in H^0(\smallcat u)$.
		
		Since $\{\smallcat u(-,U)\}_{U \in \mathcal U}$ is a set of compact objects classically generating $\dercomp(\smallcat u)$, we have that there exists a unique exact functor $L: \dercomp(\smallcat u) \to \cat T$ preserving small direct sums and sending $\smallcat u(-,U)$ to $U$ for each $U \in \mathcal U$ naturally in $U \in H^0(\smallcat u)$. We proceed to prove that $L$ is left adjoint to $H^0(G)$, namely, we want to show that for every $M \in \dercomp(\smallcat u), X \in \cat T$ there is an isomorphism
		\begin{equation*}
			\dercomp(\smallcat u)(M,H^0(G)(X)) \cong \cat T(L(M),X)
		\end{equation*}
		of abelian groups which is functorial in $M$ and $X$. Observe that, because $L$ is an exact functor preserving small direct sums and $\{\smallcat u(-,U)\}_{U \in \mathcal U}$ is a set of compact generators of $\dercomp(\smallcat u)$, we are reduced to prove that 
		\begin{equation*}
			\dercomp(\smallcat u)(\smallcat u(-,U),H^0(G)(X)) \cong \cat T(L(\smallcat u(-,U)),X)
		\end{equation*}
		for all $U \in \mathcal U$. This is indeed the case, as we have that
		\begin{align*}
			\dercomp(\smallcat u)(\smallcat u(-,U),H^0(G)(X)) &= \dercomp(\smallcat u)(\smallcat u(-,U),H^0(\cat A(j(-),X)))\\ 
			&\cong H^0(\cat A(j(U),X)) \\
			&= H^0(\cat A(U,X)) \\
			&= \cat T(L(\smallcat u(-,U)),X)
		\end{align*}
		for all $U \in \mathcal U$, which concludes the proof.
	\end{proof}
	\end{proposition}

\subsubsection{$F$ is right t-exact}

\begin{notation}\label{not:H0(F)}
	From now on, we set $\cat T= H^0(\cat A)$. For the sake of readability, we will often abuse notation and set $F=H^0(F)$, $G=H^0(G)$. We have an adjunction
	\[
	F \dashv G \colon \dercomp(\smallcat u) \rightleftarrows \cat T,
	\]
	with $G(A) = \cat A(j(-),A)$ and $F(\smallcat u(-,U)) \cong U$ for all $U \in \smallcat u$.
\end{notation}	

\begin{proposition}\label{prop:right-texact}
	$G$ is left t-exact, namely $G(\cat T_{\geq 0}) \subseteq \dercomp(\smallcat u)_{\geq 0}$. In particular, $F$ is right t-exact, namely $F(\dercomp(\smallcat u)_{\leq 0}) \subseteq \cat T_{\leq 0}$.
\end{proposition}
\begin{proof}
	Assume that $A \in \cat T_{\geq 0}$, and let $k > 0$. Then:
	\begin{align*}
		H^{-k}(G(A)) & = H^{-k}(\cat A(j(-),A)) \\
		& \cong \cat T(j(-),A[-k]) \\
		&= 0,
	\end{align*}
	since the objects of $\smallcat u$ lie in the left aisle $\cat T_{\leq 0}$ by assumption. This proves that $G(A) \in \dercomp(R)_{\geq 0}$. The right t-exactness of $F$ follows from the fact that it is the left adjoint of $G$ (Proposition \ref{prop:left-right-texact}).
\end{proof}

\subsubsection{The key lemma: A derived version of Mitchell's argument}\label{subsect:Mitchell}
The rest of the proof of the Gabriel-Popescu theorem for t-structures will be based on a derived version of a lemma by Mitchell \cite{mitchell-gabrielpopescu}, further generalized to the ``several generators'' case in \cite{crivei-mitchell-generalized}. We devote this subsection to provide this result and some of its consequences.

We start with a technical lemma, where we make careful use of functorial cones in $\cat A$.
\begin{lemma} \label{lemma:filteredcolim_functorialcones}
Let $\{X_i : i \in I \}$ be a set of objects in $H^0(\cat A)$, and let
\[
x_i \colon X_i \to \bigoplus_{i \in I} X_i
\]
be the natural morphism into the direct sum in $H^0(\cat A)$. If $F \subseteq I$ is a finite set, let
\[
x_F = \oplus_{F} x_i \colon \bigoplus_{i \in F} X_i \to \bigoplus_{i \in I} X_i,
\]
be the morphism induced by the $x_i$ for $i \in F$; if $F \subseteq F'$ is an inclusion of finite subsets of $I$, let
\[
x_{F,F'} \colon \bigoplus_{i \in F} X_i \to \bigoplus_{i \in F'} X_i
\]
the natural morphism. The morphisms $(x_{F,F'})$ clearly form a directed system, and
\begin{equation}
x_{F'} \circ x_{F,F'} = x_F
\end{equation}
in $H^0(\cat A)$. Moreover, let
\[
\psi \colon \bigoplus_{i \in I} X_i \to Y
\]
be a morphism in $H^0(\cat A)$. For any finite subset $F \subseteq I$, let
\[
\psi_F \colon \bigoplus_{i \in F} X_i \to Y
\]
be the composition $\psi \circ x_F$.

Set $K=\cone(\psi)[-1]$ and $K_F = \cone(\psi_F)[-1]$ for any finite subset $F \subseteq I$. Then, for any finite subsets $F \subseteq F' \subseteq I$, we can find morphisms
\begin{equation*}
    \begin{split}
        \alpha_F \colon K_F & \to K, \\
        \alpha_{F,F'} \colon K_F &\to K_{F'},
    \end{split}
\end{equation*}
such that the morphisms $(\alpha_{F,F'})$ form a directed system and 
\begin{equation}
\alpha_{F'} \circ \alpha_{F, F'} = \alpha_F
\end{equation}
in $H^0(\cat A)$; furthermore, these morphisms fit in the following morphisms of distinguished triangles:
\begin{equation}
    \begin{tikzcd}
K_F \arrow[r] \arrow[d, "{\alpha_{F,F'}}"] & \bigoplus_{i \in F} X_i \arrow[r, "\psi_F"] \arrow[d, "{x_{F,F'}}"] & Y \arrow[d, equal] \\
K_{F'} \arrow[d, "\alpha_{F'}"] \arrow[r]  & \bigoplus_{i \in F'} X_i \arrow[d, "x_{F'}"] \arrow[r, "\psi_{F'}"]    & Y \arrow[d, equal] \\
K \arrow[r]                                & \bigoplus_{i \in I} X_i \arrow[r, "\psi"]                           & Y.          
\end{tikzcd}
\end{equation}

Moreover, the morphisms $(\alpha_F)$ induce isomorphisms
\begin{equation}
    \varinjlim_{F \subseteq I} H^n_t(K_F) \xrightarrow{\sim} H^n_t(K)
\end{equation}
in $H^0(\cat A)^\heartsuit$, for all $n \in \mathbb Z$.
\end{lemma}
\begin{proof}
Up to quasi-equivalence, we may assume that $\cat A$ is \emph{strongly pretriangulated} (cf. \S \ref{subsubsec:dgcat_pretr}). We let $X$ be the direct sum of $\{X_i : i \in I\}$ in $H^0(\cat A)$. For all $i \in I$, we choose a closed and degree $0$ morphism
\[
x_i \colon X_i \to X  \coloneqq \bigoplus_{i \in I}X_i
\]
representing the natural inclusion in $H^0(\cat A)$. Since $\cat A$ is strongly pretriangulated, the \emph{finite} direct sums $\bigoplus_{i \in F} X_i$ exist (strictly!) in $\cat A$, and the natural morphisms
\[
x_{F, F'} \colon \bigoplus_{i \in F} X_i \to \bigoplus_{i \in F'} X_i
\]
are well defined closed degree $0$ morphisms in $\cat A$, for finite subsets $F \subseteq F'$ of $I$. Moreover, the morphisms $x_i$ induce morphisms
\[
x_F \colon \bigoplus_{i \in F} X_i \to X,
\]
and it is immediate to see that we have a \emph{strict equality}:
\[
x_{F'} \circ x_{F,F'} = x_F.
\]
Next, we take a closed degree $0$ representative of the morphism $\psi \colon X \to Y$. We abuse notation and call it again $\psi$. We then define:
\[
\psi_F = \psi \circ x_F
\]
as a closed degree $0$ morphism in $\cat A$, so that we have a (strictly!) commutative diagram of closed degree $0$ morphisms:
\[
\begin{tikzcd}
\bigoplus_{i \in F} X_i \arrow[d, "{x_{F,F'}}"] \arrow[r, "\psi_F"] & Y \arrow[d, equal] \\
\bigoplus_{i \in F'} X_i \arrow[r, "\psi_{F'}"] \arrow[d, "x_{F'}"] & Y \arrow[d, equal] \\
X \arrow[r, "\psi"]                                                 & Y,
\end{tikzcd}
\]
for finite subsets $F \subseteq F'$ of $I$. We now define $K = \cone(\psi)[-1]$ and $K_F = \cone(\psi_F)[-1]$, which exist as strict functorial cones in $\cat A$, since $\cat A$ is strongly pretriangulated. Invoking Lemma \ref{lemma:functorialcones_strict}, we find \emph{unique} closed degree $0$ morphisms
\begin{align*}
    \alpha_F \colon K_F & \to K, \\
    \alpha_{F,F'} \colon K_F & \to K_{F'},
\end{align*}
which fit in the (strictly!) commutative diagram:
\begin{equation} \label{eq:filteredcolim_functorialcones_diagram_proof} \tag{$\ast$}
\begin{tikzcd}
K_F \arrow[r] \arrow[d, dotted, "{\alpha_{F,F'}}"] & \bigoplus_{i \in F} X_i \arrow[r, "\psi_F"] \arrow[d, "{x_{F,F'}}"] & Y \arrow[d, equal] \arrow[r] & K_F[1] \arrow[d, dotted, "{\alpha_{F,F'}[1]}"] \\
K_{F'} \arrow[d, dotted, "\alpha_{F'}"] \arrow[r]  & \bigoplus_{i \in F'} X_i \arrow[d, "x_{F'}"] \arrow[r, "\psi_{F'}"]    & Y \arrow[d, equal] \arrow[r] & K_{F'}[1] \arrow[d, dotted, "{\alpha_{F'}[1]}"]  \\
K \arrow[r]                                & X \arrow[r, "\psi"]                           & Y \arrow[r] & K[1].          
\end{tikzcd}
\end{equation}
Since $x_F = x_{F'} \circ x_{F, F'}$, we conclude by uniqueness that
\[
\alpha_F = \alpha_{F'} \circ \alpha_{F,F'},
\]
for finite subsets $F \subseteq F' \subseteq I$. A similar uniqueness argument shows that the morphisms $(\alpha_{F,F'})$ form indeed a directed system. By passing to the homotopy category $H^0(\cat A)$, we get the proof of the first part of the lemma.

For the second part, we take the following commutative diagram obtained by taking cohomological long exact sequences from \eqref{eq:filteredcolim_functorialcones_diagram_proof} and then filtered colimits on the poset of finite subsets of $I$ ($n \in \mathbb Z$):
\[
\begin{tikzcd}
\displaystyle \varinjlim_{F \subseteq I}\bigoplus_{i \in F} H^n_t(X_i) \arrow[r] \arrow[d, "\varinjlim H^n_t(x_F)"] & H^n_t(Y) \arrow[r] \arrow[d, equal] & \displaystyle \varinjlim_{F \subseteq I}H^n_t(K_F) \arrow[r] \arrow[d, "\varinjlim H^n_t(\alpha_F)"] & \displaystyle \varinjlim_{F \subseteq I}\bigoplus_{i \in F} H^{n+1}_t(X_i) \arrow[r] \arrow[d, "\varinjlim H^{n+1}_t(x_F)"] & H^{n+1}_t(Y) \arrow[d, equal] \\
\displaystyle \bigoplus_{i \in I} H^n_t(X_i) \arrow[r]                                                            & H^n_t(Y) \arrow[r]           & H^n_t(K) \arrow[r]                                                                   & \displaystyle \bigoplus_{i \in I} H^{n+1}_t(X_i) \arrow[r]                                                                & H^{n+1}_t(Y).  
\end{tikzcd}
\]
We used that $H^k_t(-)$ preserves direct sums, and identified
\[
H^k_t(X) \cong \bigoplus_{i \in I} H^k_t(X_i)
\]
for all $k \in \mathbb Z$. Furthermore, it is well-known that the morphism
\[
\varinjlim_{F \subseteq I} \bigoplus_{i \in F} H^k_t(X_i) \to \bigoplus_{i \in I} H^k_t(X_i)
\]
is an isomorphism. By assumption, filtered colimits are exact in the heart: this implies that both rows of the above commutative diagram are exact, hence we may apply the five lemma and conclude that
\[
\varinjlim_{F \subseteq I} H^n_t(K_F) \to H^n_t(K)
\]
is an isomorphism, as we wanted.
\end{proof}

\begin{notation}\label{not:H0(j)}
	From now on, as done in \cref{not:H0(F)}, we will often abuse notation and set $j=H^0(j)$. We have that for every $X \in \cat T$, the functor
	\[
	\cat T(j(-),X):\opp{H^0(\smallcat u)} \to \Mod(H^0(k)),
	\]
	is an $H^0(\smallcat u)$-module, i.e. it is an object in $\Mod(H^0(\smallcat u))$.
\end{notation}
We can now prove our key lemma:
\begin{lemma}[``à la Mitchell''] \label{lemma:Mitchell_extension}
	Let $B \in \cat T_{\geq 0}$, and let $A \in \cat T$. Let $^{A}M \hookrightarrow \cat T(j(-),A)$ be a submodule (i.e. a monomorphism in $\Mod (H^0(\smallcat u))$), and let 
	\[
	f \colon ^{A}M \to \cat T(j(-),B)
	\]
	be a morphism in $\Mod (H^0(\smallcat u))$. Denote by $M$ the set $\coprod_{U \in \mathcal U}$ $^{A}M(U)$. Consider the induced diagram:
	\[
	\begin{tikzcd}
		\bigoplus_{m \in M} U_m \arrow[r, "\psi"] \arrow[d, "\varphi"'] & A \\
		B,                                                             &  
	\end{tikzcd}
	\]
	where $m \in {^{A}}M(U_m) \subseteq M$ and $\varphi$ and $\psi$ are defined as follows. If $u_m \colon U_m \to \bigoplus_{m \in M} U_m$ is the natural map associated to $m \in M$, then:
	\begin{align*}
		\psi u_m &= m \in \cat T(U_m,A), \\
		\varphi u_m &= f(U_m)(m) \in \cat T(U_m,B).
	\end{align*}
	Then, $\varphi$ factors through $\operatorname{Coim}(\psi)$:
	\begin{equation}
		\begin{tikzcd}
			\bigoplus_{m \in M} U_m \arrow[r] \arrow[d, "\varphi"'] \arrow[rr, "\psi", bend left] & \operatorname{Coim}(\psi) \arrow[r, hook] \arrow[ld, dotted] & A \\
			B.                                                                                   &                                                    &  
		\end{tikzcd}
	\end{equation}
\end{lemma}
\begin{proof}
	The proof is a generalization of Mitchell's argument \cite{mitchell-gabrielpopescu}. Let $K$ be the object sitting in the distinguished triangle
	\[
	K \to \bigoplus_{m \in M} U_m \xrightarrow{\psi} A.
	\]
	Then, denoting by $\mu \colon \tau_{\leq 0} K \to \bigoplus_{m \in M} U$ the composition
	\[
	\tau_{\leq 0} K \to K \to \bigoplus_{m \in M} U_m,
	\]
	it is enough to check that $\varphi \mu = 0$ (see Proposition \ref{prop:epimono_leftaisle}). Notice that $\psi \mu =0$.
    Let $\mathcal F$ be the poset of finite subsets of $M$, and for $F \in \mathcal F$ denote by
	\[
	\psi_F \colon \bigoplus_{m \in F} U_m \to A
	\]
	the morphism induced by $\psi$. Let $K_F = \cone(\psi_F)[-1]$. We apply Lemma \ref{lemma:filteredcolim_functorialcones} to the distinguished triangles
	\[
	K_F \to \bigoplus_{m \in F} U_m \xrightarrow{\psi_F} A
	\]
	and taking truncations we also get a directed system
	\[
	(\lambda_{F,F'} \colon \tau_{\leq 0} K_F \to \tau_{\leq 0} K_{F'})_{F,F' \in \mathcal F}
	\]
	and morphisms
	\[
	\lambda_F \colon \tau_{\leq 0} K_F \to \tau_{\leq 0} K
	\]
	such that $\lambda_F = \lambda_{F'} \circ \lambda_{F,F'}$. If we define $\mu_F$ as the composition
	\[
	\tau_{\leq 0} K_F \to K_F \to \bigoplus_{m \in F} U_m,
	\]
	we get following commutative diagram in $\cat T$ (rows are not distinguished triangles):
	\[
	\begin{tikzcd}
		\tau_{\leq 0} K_F \arrow[r, "\mu_F"] \arrow[d, "\lambda_F"] & \bigoplus_{m \in F} U_m \arrow[r, "\psi_F"] \arrow[d, "\sum_{m \in F} u_m p_m"] & A \arrow[d, equal] \\
		\tau_{\leq 0} K \arrow[r, "\mu"]                            & \bigoplus_{m \in M} U_m \arrow[r, "\psi"]                                                    & A,                               
	\end{tikzcd}
	\]
	We wrote the natural morphism $\bigoplus_{m \in F} U_m \to \bigoplus_{m \in M} U_m$ as $\sum_F u_m p_m$, where $p_m$ is the natural projection of place $m$ from $\bigoplus_{m \in F} U_m$ to $U_m$. Thanks to Lemma \ref{lemma:filteredcolim_functorialcones}, the morphisms $\lambda_F \colon \tau_{\leq 0} K_F \to \tau_{\leq 0}K$ yield an isomorphism
	\begin{equation} \label{eq:mitchell_filteredcolim}
	\varinjlim_{F \in \mathcal F} H^0(\lambda_F) \colon \varinjlim_{F \in \mathcal F} H^0_t(\tau_{\leq 0} K_F) \xrightarrow{\sim} H^0_t(\tau_{\leq 0}K),
	\end{equation}
	recalling that we may identify $H^0_t(-)$ with $H^0_t(\tau_{\leq 0}-)$.
	
	Now, $\varphi \mu$ is a morphism $\tau_{\leq 0} K \to B$ and $B \in \cat T_{\geq 0}$, so this morphism (and, in general, morphisms with codomain $B$ and domain in $\cat T_{\leq 0}$) is completely determined by its $H^0$:
	\begin{align*}
		\cat T(\tau_{\leq 0} K, B) & \cong \cat T(\tau_{\leq 0} K, \tau_{\leq 0} B) \\
		& \cong \cat T(\tau_{\geq 0} \tau_{\leq 0} K, \tau_{\leq 0} B) \\
		& \cong \cat T(H^0(\tau_{\leq 0}K), H^0(B)).
	\end{align*}
	We hence want to prove that $H^0(\varphi \mu)=0$. Thanks to the above isomorphism \eqref{eq:mitchell_filteredcolim}, it is enough to prove that $H^0(\varphi \mu \lambda_F)=0$ for all $F$. In turn, since $\mathcal U$ is a set of generators, this is equivalent to proving that $H^0(\varphi \mu \lambda_F \alpha)=0$ for all $\alpha \colon U \to \tau_{\leq 0} K_F$ with $U \in \mathcal U$, which again is equivalent to proving $\varphi \mu \lambda_F \alpha = 0$. Now, we can compute:
	\begin{align*}
		\varphi\mu\lambda_F \alpha &= \varphi \sum_F u_m p_m \mu_F \alpha \\
		&= \sum_F (f(U_m)(m)) p_m \mu_F \alpha \\
		&= \sum_F f(U)(m p_m \mu_F \alpha) \quad \text{($f$ is a morphism in $\Mod(H^0(\smallcat u)$)} \\
		&= \sum_F f(U)(\psi u_m p_m \mu_F \alpha) \\
		&= f(U)(\psi\mu \lambda_F \alpha)= \\
		&= f(U)(0) \\
		&= 0,
	\end{align*}
	and we conclude.
\end{proof}

The following easy corollary will be useful for our purposes.
\begin{corollary} \label{corollary:mitchell_technical}
	In the framework and notation of the above Lemma \ref{lemma:Mitchell_extension}, assume that we can find a further factorization of $\varphi$ through $\psi$:
	\begin{equation} \label{eq:mitchell_technical_extension}
		\begin{tikzcd}
			\bigoplus_{m \in M} U_m \arrow[r] \arrow[d, "\varphi"'] \arrow[rr, "\psi", bend left] & \operatorname{Coim}(\psi) \arrow[r, hook] \arrow[ld] & A \arrow[dll, "\tilde{\varphi}", dotted] \\
			B.                                                                                   &                                                    &  
		\end{tikzcd}
	\end{equation}
	Then, the morphism $\tilde{f} \coloneqq \tilde{\varphi}_* \colon \cat T(j(-),A) \to \cat T(j(-),B)$ makes the following diagram commute:
	\begin{equation} \label{eq:Mitchell_extension_diagram}
		\begin{tikzcd}
			^{A}M \arrow[r, hook] \arrow[d, "f"'] & {\cat T(j(-),A)} \arrow[ld, "\tilde{f}", dotted] \\
			{\cat T(j(-),B).}                     &                                              
		\end{tikzcd}
	\end{equation}
\end{corollary}
\begin{proof}
	From the commutative diagram:
	\[
	\begin{tikzcd}
		\bigoplus_{m \in M} U_m \arrow[r, "\psi"] \arrow[d, "\varphi"'] & A \arrow[ld, "\tilde{\varphi}"] \\
		B                                                            &                                
	\end{tikzcd}
	\]
	we can compute:
	\begin{align*}
		\tilde{\varphi}_*(U_m)(m) &= \tilde{\varphi} \circ m \\
		&= \tilde{\varphi} \circ \psi \circ u_m  \\
		&= \varphi \circ u_m \\
		&= f(U_m)(m),
	\end{align*}
	for all $m \in M$, where we are using that $m \in {^{A}}M(U_m) \subseteq \cat T(U_m,A)$.
\end{proof}
\begin{remark} \label{remark:extension_dginj}
    In the setting of Lemma \ref{lemma:Mitchell_extension}, assume that $B=E$ is a derived injective. Then, Corollary \ref{corollary:mitchell_technical} is applicable. Indeed, by Corollary \ref{coroll:dginj_extension_tmono} we can find an extension of $\operatorname{Coim}(\psi) \to E$ along the t-monomorphism $\operatorname{Coim}(\psi) \hookrightarrow A$.
\end{remark}
\begin{remark} \label{remark:extension_identity_tepi}
	In the setting of Lemma \ref{lemma:Mitchell_extension}, assume that $A \in \cat T_{\leq 0}$, take $^{A}M=\cat T(j(-),A)$, the morphism $\cat T(j(-),A) \to \cat T(j(-),A)$ be the identity and, as above, $M = \coprod_{U \in \mathcal U} \cat T(U,A)$. Then, $\operatorname{Coim}(\psi) \cong A$ and Corollary \ref{corollary:mitchell_technical} is applicable. Indeed, thanks to Corollary \ref{corollary:tepi_coimage} and Remark \ref{remark:tstruct_nondegerate_aisles} it is enough to check that the morphism
	\[
	\psi \colon \bigoplus_{m \in M} U_m \to A
	\]
	is surjective in $H^0$. Let $y \colon H^0(A) \to Z$ be a morphism in the heart $\cat T^\heartsuit$. We want to prove that if $y H^0(\psi)=0$, then $y=0$. Thanks to the compatibility of $H^0$ with direct sums, the assumption $y H^0(\psi)=0$ is equivalent to $y H^0(\psi u_m) =y H^0(m)=0$ for all $m \in M$, where $u_m \colon U_m \to \bigoplus_{m \in M} U_m$ is the natural morphism and by definition $\psi u_m = m \in \cat T(U_m,A)$. Next, let $p \colon \bigoplus_{i \in I} U_i \to A$ be a morphism with $U_i \in \mathcal U$ for all $i \in I$ which is surjective in $H^0$, and let $p_i \colon U_i \to A$ be the component corresponding to $i \in I$. Then, we have $y H^0(p_i)=0$ for all $i \in I$, hence $y H^0(p)=0$. We finally conclude that $y=0$ since $H^0(p)$ is an epimorphism.
\end{remark}	

We may use Lemma \ref{lemma:Mitchell_extension} and Corollary \ref{corollary:mitchell_technical} to obtain the following two corollaries that can be seen as providing a variant of fully faithfulness of the functor $G$. 
\begin{corollary} \label{coroll:fullyfaithfulness_target_dginj}
	Let $A \in \cat T$ and let $E \in \cat T_{\geq 0}$ be a derived injective. Recall that we had $j: \smallcat u =\tau_{\leq 0} \cat U \to \cat A$. Then, the morphism
	\begin{align*}
		\cat T(A,E) & \to \Mod(H^0(\smallcat u))\left(\cat T\left(j(-),A\right), \cat T\left(j(-),E\right)\right), \\
		f & \mapsto f_*
	\end{align*}
	is an isomorphism.
\end{corollary}
\begin{proof}
	First, we check injectivity. Let $f \colon A \to E$ such that $f_*=0$. Since $U \in \cat T_{\leq 0}$ for all $U \in \cat U$, we have $\cat T(j(-),B) \cong \cat T(j(-), \tau_{\leq 0} B)$ for any $B \in \cat T$, and $f_*=0$ is the same as $(\tau_{\leq 0} f)_*=0$, with
	\[
	(\tau_{\leq 0} f)_* \colon \cat T(j(-),\tau_{\leq 0} A) \to \cat T(j(-), \tau_{\leq 0} E).
	\]
	On the other hand, since $E$ is a derived injective, $f=0$ is equivalent to
	\[
	H^0(f) =0 \colon H^0(A) \to H^0(E)
	\]
	in the heart. Now, take a morphism $p \colon \bigoplus_{i \in I} U_i \to \tau_{\leq 0}A$ with $U_i \in \mathcal U$ inducing an epimorphism in $H^0$, and denote by $j_i \colon U_i \to \bigoplus_{i \in I} U_i$ the natural morphism, for $i \in I$. By hypothesis, we have that $(\tau_{\leq 0}f)_*(p j_i)=0$ for all $i \in I$, which implies $(\tau_{\leq 0}f)_*(p)=0$; in particular $H^0(f)H^0(p)=0$. Since $H^0(p)$ is an epimorphism, we conclude that $H^0(f)=0$, as we wanted.
	
	Let us now show surjectivity. We apply Lemma \ref{lemma:Mitchell_extension} taking the identity morphism $^{A}M=\cat T(j(-),A)$ and we recall Remark \ref{remark:extension_dginj}: by Corollary \ref{corollary:mitchell_technical} we actually obtain that, given any $g \colon \cat T(j(-),A) \to \cat T(j(-),E)$, there exists $\varphi \colon A \to E$ such that $\varphi_* = g$, which is precisely what we need.
\end{proof}
\begin{corollary} \label{coroll:fullyfaithfulness_target_rightaisle}
	Let $A \in \cat T_{\leq 0}$ and let $B \in \cat T_{\geq 0}$. Recall that we had $j: \smallcat u =\tau_{\leq 0} \cat U \to \cat A$. Then, the morphism
	\begin{align*}
		\cat T(A,B) & \to \Mod(H^0(\smallcat u))\left(\cat T\left(j(-),A\right), \cat T\left(j(-),B\right)\right), \\
		f & \mapsto f_*
	\end{align*}
	is an isomorphism.
\end{corollary}
\begin{proof}
	First, we check injectivity. Let $f \colon A \to B$ such that $f_*=0$. Since $A \in \cat T_{\leq 0}$ and $B \in \cat T_{\geq 0}$, we have $\cat T(A,B) \cong \cat T(H^0(A), H^0(B))$, and $f=0$ is equivalent to
	\[
	H^0(f) =0 \colon H^0(A) \to H^0(B)
	\]
	in the heart. Now, take a morphism $p \colon \bigoplus_{i \in I} U_i \to A$ with $U_i \in \mathcal U$ inducing an epimorphism in $H^0$, and denote by $j_i \colon U_i \to \bigoplus_{i \in I} U_i$ the natural morphism, for $i \in I$. By hypothesis, we have that $f_*(p j_i)=0$ for all $i \in I$, which implies that $f_*(p)=0$;  in particular $H^0(f)H^0(p)=0$. Since $H^0(p)$ is an epimorphism, we conclude that $H^0(f)=0$, as we wanted.
	
	Let us now show surjectivity. We apply Lemma \ref{lemma:Mitchell_extension} taking the identity morphism $^{A}M=\cat T(j(-),A)$ and we recall Remark \ref{remark:extension_identity_tepi}: by Corollary \ref{corollary:mitchell_technical} we actually obtain that, given any $g \colon \cat T(j(-),A) \to \cat T(j(-),B)$, there exists $\varphi \colon A \to B$ such that $\varphi_* = g$, which is precisely what we need.
\end{proof}
\begin{corollary} \label{coroll:fullyfaithfulness_heart}
The functor $G^\heartsuit: \cat T^\heartsuit \to \Mod(H^0(\smallcat u))$ is fully faithful.
\end{corollary}
\begin{proof}
 Apply \cref{coroll:fullyfaithfulness_target_rightaisle} to objects $A$ and $B$ in the heart $\cat T^\heartsuit$. One concludes immediately by recalling that $G^\heartsuit$ is precisely given by
	\[
	G^\heartsuit(X)= H^0(G(X)) = \cat T(j(-),X). \qedhere
	\]
\end{proof}
\subsubsection{$G$ preserves derived injectives}
In this subsection we make use of \S\ref{subsect:Mitchell} in order to show that $G$ preserves derived injectives. 
	\begin{proposition} \label{prop:H0G(dginj)_injective}
		Let $E \in \cat T$ be a derived injective. Then,
		\[
		H^0(G(E)) \cong \cat T(j(-),E)
		\]
		is an injective object in $\Mod(H^0(\smallcat u))$.
	\end{proposition}
	\begin{proof}
		Applying the generalization of Baer's criterion to Grothendieck categories \cite[Prop 3.2]{crivei-mitchell-generalized}, it is enough to prove that for all $U \in \smallcat u$ and any subobject $M \subseteq H^0(\smallcat u) (-,U) = \cat T(j(-),U)$, any morphism $f \colon M \to \cat T(j(-),E)$ in $\Mod(H^0(\smallcat u))$ can be extended:
		\[
		\begin{tikzcd}
			M \arrow[r, hook] \arrow[d, "f"'] & {\cat T(j(-),U)} \arrow[ld, "\tilde{f}", dotted] \\
			{\cat T(j(-),E).}                     &                                              
		\end{tikzcd}
		\]
		This follows directly from the above Corollary \ref{corollary:mitchell_technical}, choosing $A=U$ and $B=E$: indeed, since $E$ is derived injective, the extension $\operatorname{Coim}(\psi) \to E$ can be further extended along $\operatorname{Coim}(\psi) \hookrightarrow U$ (Remark \ref{remark:extension_dginj}).
	\end{proof}
	
	\begin{proposition}\label{prop:preserve-dginj}
		The functor $G \colon \cat T \to \dercomp(\smallcat u)$ preserves derived injectives.
	\end{proposition}
	\begin{proof}
		Let $E \in \cat T_{\geq 0}$ be a derived injective. We know from right t-exactness that $G(E) \in \dercomp(\smallcat u)_{\geq 0}$ and from Corollary \ref{prop:H0G(dginj)_injective} that $H^0(G(E))$ is an injective object. From \Cref{prop:dginj_brownrepresentability} we have that the derived category $\dercomp(\smallcat u)$ has derived injectives, so there is a derived injective object $K=L(H^0(G(E)))$. From the universal property, we also get a morphism
		\[
		\alpha \colon G(E) \to K,
		\]
		such that $H^0(\alpha)=1_{H^0(G(E))}$. We are going to show that $\alpha$ is an isomorphism. To do this, it is enough to show that
		\[
		\dercomp(\smallcat u)(\smallcat u(-,U)[-n], G(E)) \xrightarrow{\alpha_*} \dercomp(\smallcat u)(\smallcat u(-,U)[-n], K),
		\]
		is an isomorphism for all $U \in \smallcat U$ and all $n>0$. We have a commutative diagram:
		\[
		\begin{tikzcd}
			{\dercomp(\smallcat u)(\smallcat u(-,U)[-n], G(E))} \arrow[d, "\alpha_*"] \arrow[r, "H^0"] & {\Mod(H^0(\smallcat u))(H^{-n}(\smallcat u(-,U)), H^0(G(E)))} \\
			{\dercomp(\smallcat u)(\smallcat u(-,U)[-n], K)} \arrow[ru, "\sim"']                        &                                     
		\end{tikzcd}
		\]
		The above diagonal morphism 
		\[
		\dercomp(\smallcat u)(\smallcat u(-,U)[-n], K) \to \Mod(H^0(\smallcat u))(H^{-n}(\smallcat u(-,U)), H^0(G(E)))
		\]
		is an isomorphism, since $K=L(H^0(G(E)))$ is the derived injective associated to $H^0(G(E))$. Let us consider the horizontal morphism
		\[
		H^0 \colon \dercomp(\smallcat u)(\smallcat u(-,U)[-n], G(E)) \to \Mod(H^0(\smallcat u))(H^{-n}(\smallcat u(-,U)), H^0(G(E))).
		\]
		If we show that it is an isomorphism as well, we conclude that $\alpha_*$ is an isomorphism, as we wish. Observe that the morphism can be identified with
		\[
		H^n(G(E))(U) \cong \cat T(U[-n],E) \to \Mod(H^0(\smallcat u))\left(\cat T (j(-),U[-n]),\cat T\left(j(-),E\right)\right),
		\]
		since $H^{-n}(\smallcat u(-,U)) \cong H^{-n} \tau_{\leq 0} \cat A(j(-),U) \cong \cat T(j(-),U[-n])$. Then, we can conclude that it is an isomorphism by Corollary \ref{coroll:fullyfaithfulness_target_dginj}.
	\end{proof}

\begin{corollary} \label{coroll:leftadj_texact}
	The functor $F \colon \dercomp(\smallcat u) \to \cat T$ is t-exact.
\end{corollary}
\begin{proof}
	In \Cref{prop:right-texact} we already proved that $F$ is right t-exact. The left t-exactness follows from \Cref{prop:preserve-dginj} and \Cref{prop:adjoint-dginj}.
\end{proof}

\subsubsection{$G$ is quasi-fully faithful}
\begin{remark}\label{rem:quasifully-faithful}
	Let us distinguish again, contradicting momentarily \Cref{not:H0(F)}, between the dg functor $G: \cat A \to \dercompdg(\smallcat u)$ and the induced functor $H^0(G): \cat T \to \dercomp(\smallcat u)$. Observe that as both $\cat A$ and $\dercompdg(\smallcat u)$ are pretriangulated dg categories, in order to prove the quasi-fully faithfulness of $G$ it is enough to prove that $H^0(G)$ is fully faithful. From this moment on we abuse notations and set, once more, $G = H^0 (G)$, as we did in \Cref{not:H0(F)}. 
\end{remark}

We recall that, by Proposition \ref{prop:induced_adjunction_hearts}, the adjunction $F \dashv G$ induces an adjunction between the hearts:
		\begin{equation}
			F^\heartsuit \dashv G^\heartsuit \colon  \Mod(H^0(\smallcat u)) \rightleftarrows \cat T^\heartsuit.
		\end{equation}
		Observe that, if $X \in \cat \cat T^\heartsuit$, then by definition
		\begin{align*}
			G^\heartsuit(X) &= H^0(\cat A(j(-),X))\\
			&= \cat T(j(-),X)) \\
			& \cong \cat T^\heartsuit(H^0(j(-)), X).
		\end{align*}
	
	\begin{proposition}\label{prop:rightadj-fully-faithful}
		The functor $G$ is fully faithful.
	\end{proposition}
	\begin{proof}
		We want to show that the counit morphism
		\begin{equation} \label{eq:counit_want_iso}
			\varepsilon_A \colon FG(A) \to A
		\end{equation}
		is an isomorphism for all $A \in \cat T$. First, we notice that \eqref{eq:counit_want_iso} is an isomorphism if and only if
		\[
		H^0(F(G(A[n]))) \to H^0(A[n])
		\]
		is an isomorphism for all $n \in \mathbb Z$. Since $F$ is t-exact, it commutes with cohomologies and we may identify
		\[
		H^0(F(G(A[n]))) = F(H^0(G(A[n]))) = F^\heartsuit(H^0(G(A[n]))) .
		\]
		Moreover, notice that the natural morphism $\tau_{\leq 0}(A[n]) \to A[n]$ induces isomorphisms
		\[
		H^0(\tau_{\leq 0}(A[n])) \to H^0(A[n])
		\]
		and
		\[
		H^0(G(\tau_{\leq 0}(A[n])) = \cat T(j(-),\tau_{\leq 0}(A[n])) \xrightarrow{\sim} \cat T(j(-),A[n]) = H^0(G(A[n])).
		\]
		We are going to view them as identifications. Hence, we conclude that \eqref{eq:counit_want_iso} is an isomorphism if and only if the morphism
		\[
		F^\heartsuit(H^0(G(\tau_{\leq 0}(A[n])))) \to H^0(\tau_{\leq 0}(A[n]))
		\]
		is an isomorphism in the heart $\cat T^\heartsuit$. Hence, we may ease notation and reduce ourselves to prove that \emph{the morphism
			\begin{equation} \label{eq:counit_iso_reduced}
				e \colon F^\heartsuit(H^0(G(A))) \to H^0(A)
			\end{equation}
			is an isomorphism in $\cat T^\heartsuit$ for all $A \in \cat T_{\leq 0}$.} This morphism \eqref{eq:counit_iso_reduced} is obtained from \eqref{eq:counit_want_iso} by taking cohomology and making the identifications discussed above.
		
		Now, \eqref{eq:counit_iso_reduced} is an isomorphism if and only if, for all objects $B \in \cat T^\heartsuit$, we have that
		\begin{equation}
			e^* \colon \cat T(H^0(A),B) \to \cat T(F^\heartsuit(H^0(G(A))), B)
		\end{equation}
		is an isomorphism. Since $B \in \cat T^\heartsuit$ and $A \in \cat T_{\leq 0}$, we may further identify
		\[
		\cat T(A,B)= \cat T(H^0(A),B).
		\]
		The adjunction $F^\heartsuit \dashv G^\heartsuit$ yields an isomorphism
		\begin{align*}
			\cat T(F^\heartsuit(H^0(G(A))), B) &\cong \Mod(H^0(
			\smallcat u))(H^0(G(A)), H^0(G(B)) \\
			&= \Mod(H^0(\smallcat u))(\cat T(j(-),A), \cat T(j(-),B)),
		\end{align*}
		and moreover it is easy to check that the composition
		\[
		\cat T(A,B)= \cat T(H^0(A),B) \xrightarrow{e^*} \cat T(F^\heartsuit(H^0(G(A))), B) \cong \Mod(H^0(\smallcat u))(\cat T(j(-),A), \cat T(j(-),B))
		\]
		is actually the natural morphism which maps $f \in \cat T(A,B)$ to $f_* \colon \cat T(j(-),A) \to \cat T(j(-),B)$. Hence, if we finally show that this composition $f \mapsto f_*$ is an isomorphism, we will conclude that $e^*$ is an isomorphism, as we want. But this now follows directly from Corollary \ref{coroll:fullyfaithfulness_target_rightaisle}.
	\end{proof}

\subsubsection{Conclusion}
We are now in position to provide the proof of our main result. 
\begin{proof}[Proof of \Cref{thm:GabrielPopescu}]
	It follows from \Cref{prop:existence-left-adjoint}, \Cref{coroll:fullyfaithfulness_heart}, \Cref{coroll:leftadj_texact}, \Cref{rem:quasifully-faithful} and \Cref{prop:rightadj-fully-faithful}.
\end{proof} 

\subsection{An application: derived categories of Grothendieck abelian categories and dg-en\-hance\-ments}\label{subsec:application} In this part, we assume that $\basering k$ is an ordinary commutative ring, and as usual every category (dg or triangulated) will be implicitly taken over $\basering k$.

If $\cat T$ is a triangulated category, a \emph{dg-enhancement} of $\cat T$ is by definition a pretriangulated category $\cat A$ such that there is a triangulated equivalence $H^0(\cat A) \cong \cat T$. Many categories admit dg-enhancements, for examples derived categories $\dercomp(\mathfrak A)$ of abelian categories: a dg-enhancement of $\dercomp(\mathfrak A)$ can be described as a dg-quotient, cf. \cite{canonaco-stellari-dgenh-survey}. Triangulated categories which admit a dg-enhancement are alled \emph{algebraic}. When such enhancements exist, it is an interesting and important problem to understand whether they are \emph{unique}, in the following sense:
\begin{definition}
    Let $\cat T$ be an algebraic triangulated category. Then, $\cat T$ \emph{has a unique dg-en\-hance\-ment} if for given pretriangulated dg-categories $\cat A$ and $\cat B$ such that $H^0(\cat A) \cong H^0(\cat B) \cong \cat T$, there exists a quasi-equivalence $\cat A \cong \cat B$ (in the sense of quasi-functors, cf. \ref{subsubsec:bimod_qfun}).
\end{definition}

The uniqueness of dg-enhancements has been addressed in recent works. In \cite{canonaco-stellari-neeman-dgenh-all} uniqueness is proved for all kinds of derived categories of abelian categories; in \cite{antieau-stableenh} uniqueness results are proved for stable $\infty$-categorical enhancements. In particular, it is now known that the derived category $\dercomp(\mathfrak G)$ of a Grothendieck abelian category $\mathfrak G$ has a unique dg-enhancement: see \cite{canonoaco-stellari-dgenh-grothendieck} for the original proof.

Here, we show how Theorem \ref{thm:GabrielPopescu} can be applied to achieve a quick proof of the uniqueness of dg-enhancement in the case of derived categories of Grothendieck abelian categories. The idea of the proof is actually quite similar to the one of \cite[Theorem 6.6]{antieau-stableenh}.
\begin{theorem} \label{thm:dgenh_uniqueness_grothendieck}
    Let $\mathfrak G$ be a Grothendieck abelian category. Then, the derived category $\dercomp(\mathfrak G)$ has a unique dg-enhancement.
\end{theorem}
\begin{proof}
Let $\cat A$ and $\cat B$ be pretriangulated dg-categories such that $H^0(\cat A) \cong H^0(\cat B) \cong \dercomp(\mathfrak G)$. The derived category $\dercomp(\mathfrak G)$ has the t-structure described in Example \ref{example:dercat_Grothendieckabelian}, which in particular is generated by a chosen generator $U \in \mathfrak G$. We may use the above triangulated equivalences to transport this t-structure on $H^0(\cat A)$ and $H^0(\cat B)$, thus making $\cat A$ and $\cat B$ t-dg-categories satisfying the assumptions of Setup \ref{setup:GabrielPopescu}. In particular, we have generators $U_1$ of $\cat A$ and $U_2$  of $\cat B$ with the property that they actually live in the hearts:
\begin{align*}
    & U_1 \in H^0(\cat A)^\heartsuit, \\
    & U_2 \in H^0(\cat B)^\heartsuit.
\end{align*}
It is immediate to observe that $\cat A(U_1,U_1)$ and $\cat B(U_2,U_2)$ are complexes concentrated in nonnegative degrees. Hence, we have that
\begin{align*}
    R_1 & := \tau_{\leq 0} \cat A(U_1,U_1) \cong H^0(\cat A)(U_1,U_1), \\
    R_2 & := \tau_{\leq 0} \cat B(U_2,U_2)  \cong H^0(\cat B)(U_2,U_2),
\end{align*}
and in particular $R_1$ and $R_2$ are ordinary rings. If we call $T \colon H^0(\cat A) \to H^0(\cat B)$ the equivalence supplied by the hypothesis, we have by construction that it is t-exact, $T(U_1)=U_2$ and there is an isomorphism
\[
f \colon R_1 = H^0(\cat A)(U_1, U_1) \xrightarrow{\sim} H^0(\cat B)(U_2,U_2) = R_2
\]
induced by $T$.

The derived Gabriel-Popescu theorem (Theorem \ref{thm:GabrielPopescu}) yields the following diagram of quasi-functors:
\[
\begin{tikzcd}
\ker(F_1) \arrow[r] & \dercompdg(R_1) \arrow[r, "F_1"] \arrow[d, shift right, "\Ind_f"'] & \cat A \\
\ker(F_2) \arrow[r] & \dercompdg(R_2) \arrow[u, shift right, "\Ind_{f^{-1}}"'] \arrow[r, "F_2"] & \cat B.
\end{tikzcd}
\]
    $F_1$ and $F_2$ are t-exact and realize $\cat A$ and $\cat B$ as dg-quotients of $\dercompdg(R_1)$ and $\dercompdg(R_2)$ respectively by $\ker(F_1)$ and $\ker(F_2)$. The dg-functors $\Ind_{f}$ and $\Ind_{f^{-1}}$ are dg-equivalences inverse to each other (cf. \S \ref{subsubsec:ind_res}). Moreover, we have that $\Ind_{f^{-1}} \cong \Res_f$ and $\Ind_f \cong \Res_{f^{-1}}$, and it is immediate to prove that they are both t-exact with respect to the canonical t-structures. If we could show that $\Ind_f$ and $\Ind_{f'}$ restrict to equivalences
    \[
    \Ind_{f} \colon \ker(F_1) \rightleftarrows \ker(F_2) : \Ind_{f^{-1}}
    \]
    we could conclude that $\cat A$ is quasi-equivalent to $\cat B$, thanks to the universal property of the dg-quotient (cf. \S \ref{subsubsec:dgquotients_localizations}).
    
    We then go on to prove that $\Ind_{f}(\ker(F_1)) \subseteq \ker(F_2)$. The reverse inclusion will follow from $\Ind_{f^{-1}}(\ker(F_2)) \subseteq \ker(F_1)$, which can be proved in the same way. For the sake of simplicity, we abuse notation and we identify $F_i=H^0(F_i)$ ($i=1,2$). Let $X \in \ker(F_1)$. We want to prove that $F_2(\Ind_f(X))\cong 0$ in $H^0(\cat B)$. The t-structure on $\cat B$ is non-degenerate and both $F_2$ and $\Ind_f$ are t-exact, hence this is equivalent to:
    \begin{equation} \label{eq:dgenh_uniqueness_proof_kernels}
    H^k(F_2(\Ind_f(X))) \cong F_2(\Ind_f(H^k(X))) \cong 0
    \end{equation}
    in $H^0(\cat B)^\heartsuit$, for all $k \in \mathbb Z$.
    
    We now consider the following diagram of Grothendieck (in particular cocomplete) abelian categories and exact functors:
    \[
    \begin{tikzcd}
    \Mod(R_1) \arrow[d, "\Ind_f"] \arrow[r, "F_1"] & H^0(\cat A)^\heartsuit \arrow[d, "T"] \\
    \Mod(R_2) \arrow[r, "F_2"]                     & H^0(\cat B)^\heartsuit.               
    \end{tikzcd}
    \]
    We can prove that this diagram is actually commutative. Since $F_1, F_2$ are cocontinuous (being left adjoints) and both $\Ind_f$ and $T$ are also cocontinuous (being equivalences), to check that $T \circ F_1 \cong F_2 \circ \Ind_f$ as functors $\Mod(R_1) \to H^0(\cat B)^\heartsuit$ is equivalent to check that they are equivalent on the generator $R_1$ of $\Mod(R_1)$. Indeed:
    \begin{align*}
        T(F_1(R_1)) & \cong T(U_1) \\
         & \cong U_2 \\
         & \cong F_2(R_2) \\
         & \cong F_2(\Ind_f(R_1)),
    \end{align*}
    and we can directly check that this isomorphism is actually natural, namely, compatible with morphisms $R_1 \to R_1$ in $\Mod(R_1)$ (it follows from the Yoneda lemma and how these functors are defined).
    
    Finally, we can check \eqref{eq:dgenh_uniqueness_proof_kernels} as follows:
    \begin{align*}
        F_2(\Ind_f(H^k(X))) & \cong T(F_1(H^k(X))) \\
        & \cong H^k(T(F_1(X))) && \text{(t-exactness)} \\
        & \cong 0 && \text{($X \in \ker(F_1)$).}
    \end{align*}
    This allows us to conclude.
\end{proof}

\appendix 

\section{A ``Baer criterion'' for derived injectivity} \label{appendix:dginj}
In this appendix, we prove a characterization of derived injectivity which holds under some additional assumptions on the given t-structure. In particular, we will obtain a ``Baer-like'' criterion for derived injectives in the derived category $\dercomp(R)$ of a $\basering k$-dg-algebra $R$ which is concentrated in nonpositive degrees (cf. Proposition \ref{prop:dercomp_naturaltstruct} for the natural t-structure on $\dercomp(R)$).

We first recall a (well-known) result which characterizes injective objects in the heart in terms of a vanishing condition.
\begin{lemma} \label{lemma:injective_vanishing_triangulated}
    Let $\cat T$ be a triangulated category with a t-structure, and let $I \in \cat T^\heartsuit$ be an object in the heart. Then, $I$ is injective if and only if
    \[
    \cat T(Z[-1],I) \cong 0,
    \]
    for all $Z \in \cat T^\heartsuit$.
\end{lemma}
\begin{proof}
Assume that the above vanishing condition holds, and let
\[
0 \to A \to B \to C \to 0
\]
be a short exact sequence in $\cat T^\heartsuit$. This induces a distinguished triangle
\[
A \to B \to C
\]
in $\cat T$. We consider the long exact sequence:
\[
\cat T(A[1], I) \to \cat T(B,I) \to \cat T(C,I) \to \cat T(A,I) \to \cat T(C[-1], I).
\]
Since $A[1] \in \cat T_{\leq -1}$, we have $\cat T(A[1], I) \cong 0$; moreover, the hypothesis ensures that $\cat T(C[-1], I) \cong 0$. We conclude that $\cat T^\heartsuit(-,I)$ is exact, as we wanted.

On the other hand, let $I$ be injective, $Z \in \cat T^\heartsuit$ and let $f \in \cat T(Z[-1],I)$. We want to prove that $f=0$. We take the distinguished triangle:
\begin{equation} \label{eq:injective_vanishing_triangulated_disttria}
Z[-1] \xrightarrow{f} I \xrightarrow{g} A \to Z. \tag{$\ast$}
\end{equation}
Since $I$ and $Z$ are in $\cat T^\heartsuit$, the same is true for $A$ (the heart is extension-closed in $\cat T$). Moreover, the sequence
\[
0 \to I \to A \to Z \to 0
\]
is exact in $\cat T^\heartsuit$. Since $I$ is injective, the morphism
\[
g^* \colon \cat T(A, I) \twoheadrightarrow \cat T(I,I)
\]
is surjective. We now take the long exact sequence associated to \eqref{eq:injective_vanishing_triangulated_disttria}:
\[
\cat T(A, I) \xrightarrow{g^*} \cat T(I,I) \xrightarrow{f^*} \cat T(Z[-1],I).
\]
We have that $f=f^*(1_I)$, and by surjectivity of $g^*$ we have that $1_I = g^*(\alpha)$ for some $\alpha \in \cat T(A,I)$. By exactness, we conclude that $f = f^* g^*(\alpha) =0$, as claimed.
\end{proof}

Next, we prove a refined version of Proposition \ref{prop:dginj_equivalent_definitions}, under the additional assumption that the t-structure is non-degenerate and its heart has a set of generators.
\begin{proposition} \label{prop:dginj_test_generators_heart}
    Let $\cat T$ be a triangulated category with a non-degenerate t-structure (Definition \ref{def:tstruct_nondegenerate}) which has derived injectives (Definition \ref{def:dginj}). Moreover, assume that the heart $\cat T^\heartsuit$ has a set $\mathcal G$ of generators. Then, for any object $E \in \cat T$, the following conditions are equivalent:
    \begin{enumerate}
        \item \label{item:dginj_test_generators_heart_1} $E$ is derived injective (Definition \ref{def:dginj}).
        \item \label{item:dginj_test_generators_heart_2} $E \in \cat T_{\geq 0}$ and for any $Z \in \cat T_{\geq 0}$ we have
	   \begin{equation*}
	       \cat T(Z[-1],E) \cong 0.
	   \end{equation*}
	    \item \label{item:dginj_test_generators_heart_3} $E \in \cat T_{\geq 0}$ and for any $Z \in \cat T^\heartsuit$ we have
	    \begin{equation}
	        \cat T(Z[-n], E) \cong 0,
	    \end{equation}
	    for all $n\geq 1$.
        \item \label{item:dginj_test_generators_heart_4} $E \in \cat T_{\geq 0}$, its zeroth cohomology $H^0(E)$ is an injective object in $\cat T^\heartsuit$ and for any generator $G \in \mathcal G$ in the heart we have:
        \begin{equation}
            \cat T(G[-n], E) \cong 0,
        \end{equation}
        for any $n \geq 1$.
    \end{enumerate}
\end{proposition}
\begin{proof}
The equivalence $\eqref{item:dginj_test_generators_heart_1} \Leftrightarrow \eqref{item:dginj_test_generators_heart_2}$ is proved in Proposition \ref{prop:dginj_equivalent_definitions}.

$\eqref{item:dginj_test_generators_heart_2} \Rightarrow \eqref{item:dginj_test_generators_heart_3}$ is immediate.

The only nontrivial part of $\eqref{item:dginj_test_generators_heart_3} \Rightarrow \eqref{item:dginj_test_generators_heart_4}$ is to prove that $H^0(E)$ is an injective object in $\cat T^\heartsuit$. To do so, we can argue as in the proof of $\eqref{item:dginj_equivalent_definitions_2} \Rightarrow \eqref{item:dginj_equivalent_definitions_1}$ of Proposition \ref{prop:dginj_equivalent_definitions}. A slightly different argument is as follows. Since $E \in \cat T_{\geq 0}$ and hence $H^0(E) \cong \tau_{\leq 0} E$, we take the canonical distinguished triangle:
\[
H^0(E) \to E \to \tau_{\geq 1} E,
\]
and for $Z \in \cat T^\heartsuit$ the induced long exact sequence:
\[
\cat T(Z[-1], (\tau_{\geq 1} E)[-1]) \to \cat T(Z[-1], H^0(E)) \to \cat T(Z[-1], E).
\]
We have $\cat T(Z[-1], (\tau_{\geq 1} E)[-1]) \cong \cat T(Z, \tau_{\geq 1} E) \cong 0$, and $T(Z[-1], E) \cong 0$ by hypothesis. We conclude that
\[
\cat T(Z[-1], H^0(E)) \to 0
\]
is a monomorphism, therefore $\cat T(Z[-1], H^0(E)) \cong 0$. We now conclude invoking the above Lemma \ref{lemma:injective_vanishing_triangulated}.

We now turn to $\eqref{item:dginj_test_generators_heart_4} \Rightarrow \eqref{item:dginj_test_generators_heart_1}$, which is the actual nontrivial part of the proof (and the only one which uses the assumptions that the t-structure is non-degenerate and has derived injectives). We shall write $[-,-]$ for the hom-space in $\cat T$. By hypothesis, $H^0(E)$ is injective and the associated derived injective $L(H^0(E))$ exists, hence we can uniquely lift the identity morphism $H^0(E) \to H^0(E)$ to a morphism in $\cat T$:
    \begin{equation}
	f \colon E \to L(H^0(E)).
	\end{equation}
The goal is to show that $f$ is an isomorphism. We already know (Proposition \ref{prop:dginj_equivalent_definitions}) that $H^0(f) = \tau_{\leq 0} f = \operatorname{id} \colon H^0(E) \to H^0(E)$. We shall now argue by induction: for $n > 0$, assume that
	\[
	\tau_{\leq n-1} f \colon \tau_{\leq n-1} E \to \tau_{\leq n-1} L(H^0(E))
	\]
is an isomorphism. We want to prove that
	\[
	H^n(f) \colon H^n(E) \to H^n(L(H^0(E)))
	\]
is also an isomorphism, which will imply that $\tau_{\leq n} f$ is an isomorphism. For simplicity, call $L(H^0(E))=E'$ and consider the following commutative diagram, where the rows are canonical distinguished triangles:
\begin{equation} \label{diagram:compare_E_otherdginj_triangles}
	\begin{tikzcd}
		{(\tau_{\leq n-1}E)[n]} \arrow[r] \arrow[d, "\sim"'] & {(\tau_{\leq n} E)[n]} \arrow[r] \arrow[d] & H^n(E) \arrow[r, "\lambda"] \arrow[d, "H^n(f)"] & {(\tau_{\leq n-1}E)[n+1]} \arrow[d, "\sim"] \\
		{(\tau_{\leq n-1}E')[n]} \arrow[r]            & {(\tau_{\leq n}E')[n]} \arrow[r]    & H^n(E') \arrow[r, "\mu"]                 & {(\tau_{\leq n-1}E')[n+1].}          
	\end{tikzcd}
\end{equation}
	
Next, let $Z \in \cat T^\heartsuit$. We first prove that
\[
    \mu_* \colon [Z,H^n(E')] \to [Z,(\tau_{\leq n-1} E')[n+1] ]
\]
is an isomorphism. Indeed, take the canonical distinguished triangle:
\[
\tau_{\leq n-1} E' \to E' \to \tau_{\geq n} E'
\]
and the induced long exact sequence:
\[
	[Z, E'[n] ] \to [Z, (\tau_{\geq n} E') [n]] \to [Z, (\tau_{\leq n-1} E')[n+1]] \to [Z, E'[n+1]].
\]
$E'$ is a derived injective, so $[Z,E'[k]]=[Z[-k],E']=0$ for $k=n,n+1$ (recall that $n>0$), so we have that
\[
	[Z, (\tau_{\geq n} E') [n]] \xrightarrow{\sim} [Z, (\tau_{\leq n-1} E')[n+1]]
\]
is an isomorphism. Finally, since $Z$ is in the heart, we also have an isomorphism
\[
	[Z, \tau_{\leq 0} ((\tau_{\geq n} E') [n])] \xrightarrow{\sim} [Z, (\tau_{\geq n} E') [n]],
\]
and then we recall that $\tau_{\leq 0} ((\tau_{\geq n} E') [n]) \cong (\tau_{\leq n} \tau_{\geq n} E')[n] = H^n(E')$, and we conclude observing that the composition
\[
    [Z, H^n(E')] \xrightarrow{\sim } [Z, (\tau_{\geq n} E') [n]] \xrightarrow{\sim} [Z, (\tau_{\leq n-1} E')[n+1]]
\]
is indeed induced by $\mu$.
		
Now, let $G \in \mathcal G$ be a generator of $\cat T^\heartsuit$ and consider the following commutative diagram (compare with \eqref{diagram:compare_E_otherdginj_triangles}):
\begin{equation}
	\begin{tikzcd}
		{[G,H^n(E)]} \arrow[r, "\lambda_*"] \arrow[d, "H^n(f)_*"] & {[G, (\tau_{\leq n-1} E)[n+1]]} \arrow[d, "\sim"] \\
		{[G,H^n(E')]} \arrow[r, "\mu_*"]                          & {[G, (\tau_{\leq n-1} E')[n+1]].}                 
			\end{tikzcd}
\end{equation}
Since $G \in \mathcal G$ is an arbitrary generator, in order to prove that $H^n(f)$ is an isomorphism it is enough to prove that $H^n(f)_*$ above is an isomorphism. We have just proved that $\mu_*$ is an isomorphism, so we are left to check that $\lambda_*$ is an isomorphism. To do so, consider the long exact sequence:
\[
	[G, (\tau_{\leq n}E)[n]] \to [G, H^n(E)] \xrightarrow{\lambda_*} [G, (\tau_{\leq n-1} E)[n+1]] \to [G, (\tau_{\leq n}E)[n+1]].
\]
We need to show that both $[G, (\tau_{\leq n}E)[n]]=0$ and $[G, (\tau_{\leq n}E)[n+1]]=0$. First, notice that
\[
	[G, (\tau_{\leq n}E)[n]] \cong [G[-n], \tau_{\leq n} E] \cong [G[-n], E] = 0
\]
by hypothesis.  Next, Consider the canonical distinguished triangle:
\[
	\tau_{\leq n} E \to E \to \tau_{\geq n+1} E,
\]
and take the induced long exact sequence:
\[
	[G, (\tau_{\geq n+1}E)[n]] \to [G, (\tau_{\leq n}E) [n+1]] \to [G, E[n+1]].
\]
By assumption, $[G, E[n+1]]=0$, and
\[
	[G, (\tau_{\geq n+1}E)[n]] \cong [G[-n], \tau_{\geq n+1} E] = 0,
\]
since $G[-n] \in \cat T_{\leq n}$ and  $\tau_{\geq n+1} E \in \cat T_{\geq n+1}$. We finally conclude that
\[
	[G, (\tau_{\leq n}E)[n+1]]=0,
\] 
as claimed, which shows that $\lambda_*$ is an isomorphism as desired. The proof is now complete.
\end{proof}

Let us now have a closer look at the derived injectives of the derived category $\dercomp(R)$, where $R$ is a fixed $\basering k$-dg-algebra concentrated in nonpositive degrees. We first discuss a variant of Lemma \ref{lemma:injective_vanishing_triangulated} which follows from the classical Baer criterion for injectivity \cite{baer-injectivity}.
\begin{lemma} \label{lemma:injective_vanishing_derivedcategory}
Let $I \in \dercomp(R)^\heartsuit = \Mod(H^0(R))$. Then, $I$ is injective if and only if
\[
\dercomp(R)(Z[-1],I) \cong 0,
\]
for all $Z \in \dercomp(R)^\heartsuit$ which admit an epimorphism $H^0(R) \twoheadrightarrow Z$.
\end{lemma}
\begin{proof}
The necessity of the condition is a particular case of the above Lemma \ref{lemma:injective_vanishing_triangulated}.

On the other hand, we take an object $I \in \dercomp(R)^\heartsuit$ which satisfies the above condition. We may check injectivity of $I$ using the Baer criterion. Indeed, it is enough to prove that, given an exact sequence
\[
0 \to B \to H^0(R) \to Z \to 0,
\]
in $\Mod(H^0(R))$, the induced morphism
\[
\Mod(H^0(R))(H^0(R), I) \to \Mod(H^0(R))(B,I)
\]
is surjective. The above short exact sequence induces a distinguished triangle
\[
Z[-1] \to B \to H^0(R) \to Z.
\]
We take the induced long exact sequence:
\[
\dercomp(R)(H^0(R), I) \to \dercomp(R)(B,I) \to \dercomp(R)(Z[-1],I).
\]
By hypothesis, $\dercomp(R)(Z[-1],I) \cong 0$, and since we may identify
\begin{align*}
    \dercomp(R)(B,I) & \cong \Mod(H^0(R))(B,I), \\
    \dercomp(R)(H^0(R), I) & \cong \Mod(H^0(R))(H^0(R), I),
\end{align*}
we conclude.
\end{proof}

We can now prove a result which resembles a ``Baer-like'' criterion for derived injectivity:
\begin{corollary}[``Baer criterion'' for derived injectivity] \label{corollary:derivedBaer}
Let $E \in \dercomp(R)$. The following are equivalent:
	\begin{enumerate}
		\item \label{item:Baer1} $E$ is derived injective (Definition \ref{def:dginj}).
		\item  \label{item:Baer2} $E \in \dercomp(R)_{\geq 0}$ and for all objects $Z \in \dercomp(R)^\heartsuit$ which admit an epimorphism $H^0(R) \twoheadrightarrow Z$, we have:
		\begin{equation}
			\dercomp(R)(Z[-n], E) =0
		\end{equation}
			for all $n>0$.
			\item \label{item:Baer3}  $E \in \dercomp(R)_{\geq 0}$ and for all objects $Z \in \dercomp(R)^\heartsuit$ which admit an epimorphism $H^0(R) \twoheadrightarrow Z$, we have:
		\begin{equation}
			\begin{split}
				\dercomp(R)(Z[-1], E) &=0, \\
				\dercomp(R)(H^0(R)[-n], E) &=0, \quad n > 1.
			\end{split}
		\end{equation}
	\end{enumerate}
\end{corollary}
\begin{proof}
$\eqref{item:Baer1} \Rightarrow \eqref{item:Baer2}$ is a particular case of Proposition \ref{prop:dginj_test_generators_heart}; $\eqref{item:Baer2} \Rightarrow \eqref{item:Baer3}$ is straightforward (of course, $H^0(R)$ admits an epimorphism $H^0(R) \to H^0(R)$).

Let us now prove $\eqref{item:Baer3} \Rightarrow \eqref{item:Baer1}$. From the condition $\dercomp(R)(Z[-1], E) \cong 0$ (for $Z \in \dercomp(R)^\heartsuit$ which admits an epimorphism $H^0(R) \twoheadrightarrow Z$) we can prove that $H^0(E)$ is injective. Indeed, we take the canonical distinguished triangle:
\[
H^0(E) \to E \to \tau_{\geq 1} E,
\]
and for $Z \in \cat T^\heartsuit$ which admits an epimorphism $H^0(R) \twoheadrightarrow Z$ we have the induced long exact sequence:
\[
\cat T(Z[-1], (\tau_{\geq 1} E)[-1]) \to \cat T(Z[-1], H^0(E)) \to \cat T(Z[-1], E).
\]
We have $\cat T(Z[-1], (\tau_{\geq 1} E)[-1]) \cong \cat T(Z, \tau_{\geq 1} E) \cong 0$, and $T(Z[-1], E) \cong 0$ by hypothesis. We conclude that
\[
\cat T(Z[-1], H^0(E)) \to 0
\]
is a monomorphism, therefore $\cat T(Z[-1], H^0(E)) \cong 0$. Then, we can invoke the above Lemma \ref{lemma:injective_vanishing_derivedcategory} and conclude that $H^0(E)$ is injective, as claimed. Now, the result easily follows from Proposition \ref{prop:dginj_test_generators_heart}.
\end{proof}
\begin{remark}
The Baer criterion characterizing derived injectives in $\dercomp(R)$ (cf. \Cref{corollary:derivedBaer}) can be easily generalized to the ``several objects'' case, that is, to a criterion for derived injectivity in $\dercomp(\smallcat a)$, where $\smallcat a$ is a small $\basering k$-linear dg-category concentrated in nonpositive degrees. The proof is parallel to the one we have provided for the single-object case; notice that in the proof of \Cref{lemma:injective_vanishing_derivedcategory} one needs to replace the classical Baer criterion for categories of modules over a ring with its generalization to any Grothendieck category \cite[Prop 3.2]{crivei-mitchell-generalized}.
\end{remark}

\bibliographystyle{amsplain}
\bibliography{biblio}

\end{document}